\RequirePackage{fix-cm}
\documentclass[smallextended]{svjour3}       
\smartqed  

\usepackage[english]{babel}
\usepackage{enumerate}
\usepackage{color}
\usepackage{subfigure}
\usepackage{dsfont}
\usepackage{graphicx}
\usepackage{epstopdf}
\usepackage{amsmath}
\usepackage{mathtools}
\allowdisplaybreaks

\usepackage{amstext}
\usepackage{amssymb}
\usepackage{mathrsfs}

\usepackage{savesym}
\savesymbol{AND}
\savesymbol{OR}
\savesymbol{NOT}
\savesymbol{TO}
\savesymbol{COMMENT}
\savesymbol{BODY}
\savesymbol{IF}
\savesymbol{ELSE}
\savesymbol{ELSIF}
\savesymbol{FOR}
\savesymbol{WHILE}

\usepackage{algorithm}
\usepackage{algorithmic}

\newcommand{\R}{\mathbb{R}}

\newcommand{\Lip}{\mathop{\rm Lip}}
\newcommand{\reg}{\mathop{\rm reg}}

\newcommand{\argmax}{\mathop{\rm arg\, max}}
\newcommand{\argmin}{\mathop{\rm arg\, min}}

\newcommand{\hx}{\hat{x}}

\newcommand{\tpi}{{\tilde \pi}}
\newcommand{\hQ}{{\hat Q}}
\newcommand{\hmu}{{\hat \mu}}

\newcommand{\sX}{{\mathsf X}}

\newcommand{\sA}{{\mathsf A}}

\newcommand{\sE}{{\mathsf E}}

\newcommand{\sU}{{\mathsf U}}

\newcommand{\F}{{\mathcal F}}

\newcommand{\C}{{\mathcal C}}

\newcommand{\Pnew}{{\mathcal P}}

\newtheorem{definition}{Definition}
\newtheorem{theorem}{Theorem}
\newtheorem{corollary}{Corollary}
\newtheorem{proposition}{Proposition}
\newtheorem{lemma}{Lemma}
\newtheorem{remark}{Remark}
\newtheorem{assumption}{Assumption}
\allowdisplaybreaks

\allowdisplaybreaks

\begin{document}

\title{Q-Learning in Regularized Mean-field Games\thanks{Corresponding author N.~Saldi.}
}


\author{Berkay Anahtarci         \and
        Can Deha Kariksiz	\and
        Naci Saldi
}


\institute{B. Anahtarci \at
              \"{O}zye\u{g}in University, \c{C}ekmek\"{o}y, \.{I}stanbul, Turkey \\
              \email{berkay.anahtarci@ozyegin.edu.tr}           
           \and
           C. D. Kariksiz \at
           \"{O}zye\u{g}in University, \c{C}ekmek\"{o}y, \.{I}stanbul, Turkey \\
           \email{deha.kariksiz@ozyegin.edu.tr}    
           \and
           N. Saldi \at
           Bilkent University, \c{C}ankaya, Ankara, Turkey \\
           \email{naci.saldi@bilkent.edu.tr}    
}

\date{Received: date / Accepted: date}

\maketitle

\begin{abstract}
In this paper, we introduce a regularized mean-field game and study learning of this game under an infinite-horizon discounted reward function. Regularization is introduced by adding a strongly concave regularization function to the one-stage reward function in the classical mean-field game model. We establish a value iteration based learning algorithm to this regularized mean-field game using fitted Q-learning. The regularization term in general makes reinforcement learning algorithm more robust to the system components. Moreover, it enables us to establish error analysis of the learning algorithm without imposing restrictive convexity assumptions on the system components, which are needed in the absence of a regularization term.   
\end{abstract}

\section{Introduction}
\label{intro}
This paper deals with the learning of regularized mean-field games (MFGs) under an infinite-horizon discounted reward function.  Regularization is introduced by adding a strongly concave regularization function to the one-stage reward function in the classical mean-field game model. In this model, a single agent interacts with a huge population of other agents and compete with the collective behaviour of them through a mean-field term, which converges to the distribution of a single generic agent as the number of agents is taken to infinity. In the limiting case, a generic agent faces a single-agent stochastic control problem with a constraint on the state distribution at each time step. This condition specifies that the state distribution should be consistent with the behaviour of the total population. In other words, at each time step, the resulting distribution of the state of each agent is the same as the flow of the state distribution when the generic agent applies this policy. This stability condition between policy and state distribution flow is called the \textit{mean-field equilibrium}.

The theory of MFGs has emerged in the work of Lasry and Lions \cite{LaLi07}, where the standard terminology of mean-field games was introduced, and independently as stochastic dynamic games by Huang, Malham\'{e} and Caines \cite{HuMaCa06}, both considering continuous time non-cooperative differential games with large but finite number of asymptotically negligible anonymous agents in interaction along with their infinite limits to establish approximate Nash equilibria. In continuous-time differential games, characterization of the mean-field equilibrium is given by a 
coupled Hamilton-Jacobi-Bellman (HJB) equation and a Kolmogorov-Fokker-Planck (FPK) equation. We refer the reader to \cite{HuCaMa07,TeZhBa14,Hua10,BeFrPh13,Ca11,CaDe13,GoSa14,MoBa16} for studies of continuous-time mean-field games with different models and cost functions, such as games with major-minor players, risk-sensitive games, games with Markov jump parameters, and LQG games.

In comparison with the continuous-time framework, there are comparably fewer results available on discrete-time mean-field games in the literature. These works have mainly studied the settings where the state space is a discrete (finite or countable) set and the agents are only coupled by their cost functions; that is, the mean-field term does not influence the evolution of the agents' states. In \cite{GoMoSo10}, a mean-field game model with finite state is studied, and \cite{AdJoWe15} considers discrete-time mean-field games with an infinite-horizon discounted cost criterion over unbounded state spaces. Discrete-time mean-field games with linear state dynamics are studied in \cite{ElLiNi13,MoBa15,NoNa13,MoBa16-cdc}. 
References \cite{Bis15,Wie19,WiAl05,Sal19} study discrete-time mean-field games subject to the average cost optimality criterion. In \cite{SaBaRaMOR2}, authors consider a discrete-time risk-sensitive mean-field game with Polish state and action spaces. References \cite{SaBaRaSIAM,SaBaRaMOR1} consider a discrete-time mean-field game with Polish state and action spaces under the discounted cost optimality criterion for both the fully-observed case and the partially-observed case, respectively.

We note that the aforementioned papers, except linear models, mostly identify the existence of mean-field equilibrium and do not propose any algorithm with convergence guarantee to compute the mean-field equilibrium. In our recent work \cite{AnKaSa19}, this problem is explored for mean-field games with abstract state and action spaces under both discounted cost and average cost criteria, where we develop a value iteration algorithm and prove that this algorithm converges to the mean-field equilibrium. In \cite{AnKaSa19-b}, we generalize this value iteration algorithm to the model-free setting by using fitted Q-learning \cite{AnMuSz07}, which is preferred over a classical Q-learning algorithm since the action space is assumed to be a compact and convex subset of a finite dimensional Euclidean space. In order to establish the contractiveness of the optimality operator in this case, one needs to prove that the optimal policy is Lipschitz continuous with respect to the current mean-field term, since the optimal policy corresponding to the current mean-field term affects the next mean-field term in the value iteration algorithm. Although establishing the Lipschitz continuity of the optimal value function with respect to the mean-field term is straightforward, it is quite challenging to do the same for the optimal policy. To overcome this challenge, it was assumed in \cite[Assumption-2.1(d)]{AnKaSa19-b} that the function in the optimality equation is strongly convex and has Lipschitz continuous gradient, which restricts the applicability of the results. Moreover, as a result of these restrictive conditions, the proof of the contraction of the mean-field equilibrium operator is much more involved. Our novel approach in this paper is to introduce a strongly convex regularization function in the one-stage reward, which helps us to obtain Lipschitz continuity of the optimal policy with respect to the mean-field term via duality between strong convexity and smoothness, and generalize the results in \cite{AnKaSa19-b}. This allows us to significantly relax the assumptions on the system components and improve the theoretical analysis. In particular, as opposed to the unregularized case, we eliminate the need for strong convexity and smoothness assumptions on the system components when establishing the Lipschitz regularity of the optimal policy with respect to the mean-field term.

In the literature, the existence of mean-field equilibria has been established for discrete-time mean-field games under the discounted optimality criterion in \cite{SaBaRaSIAM}. However, learning discrete-time mean-field games has not been studied much, even for the classical case, until recently. In \cite{GuHuXuZh19}, authors establish a Q-learning algorithm to compute approximate mean-field equilibria for finite state-action mean-field games, where the convergence of the learning algorithm is dependent upon the assumption that the operators in the algorithm are contractive. In \cite{ElPeLaGePi19}, authors develop a fictitious play iterative learning algorithm for mean-field games with compact state and action spaces, where the dynamics of the state and the one-stage cost function satisfy certain structure, and suggest an error analysis of the learning algorithm for the deterministic game model (no noise term in the state dynamics). In \cite{CaLaTa19} authors study linear-quadratic mean-field games and establish the convergence of policy gradient algorithm. In \cite{FuYaChWa19}, an actor-critic algorithm to learn mean-field equilibrium for linear-quadratic mean-field games is developed. In \cite{YaYeTrXuZh18} a mean-field game in which agents can control their transition probabilities without any restriction is studied. In this case, the action space becomes the set of probability measures on the state space, and the authors are able to transform a mean-field game into an equivalent deterministic Markov decision process by extending the state and action spaces, establishing classical reinforcement learning algorithms to compute mean-field equilibrium. In the continuous-time setup, the following early reference \cite{YiMeMeSh14}  develops a learning algorithm for mean-field oscillator game model to obtain approximate Nash equilibrium (see also Example in \cite[Section IV-C]{MeMe10} for learning algorithm developed for continuous-time LQG mean-field game problem).

In misspecified control models, greedy algorithms often results in policies that are far from optimal. Our approach of making use of regularization also provides a way to overcome this problem. Most recent reinforcement algorithms use regularization to increase exploration and robustness, and the regularization is generally established via entropy or relative entropy. We refer the reader to \cite{GeScPi19} for an exhaustive review of the literature on regularized Markov decision processes (MDPs) and \cite{NeJoGo17} for a general framework on entropy-regularized MDPs. In this paper, we introduce regularized mean-field games, analogous to regularized MDPs. Our research seems to be the first one studying this problem. We propose a learning algorithm to compute an equilibrium solution for discrete-time regularized mean-field games under the discounted reward optimality criterion. A regularization term is added to the one-stage reward function in this game model, making the algorithm more robust since one can establish the Lipschitz sensitivity of the optimal policy to the system components using duality between strong convexity and smoothness, which is a common necessity in robustness analysis \cite[Remark 4.3]{KaYu20},\cite[Theorem 4.1]{KaYu19}. As mentioned above, regularization additionally provides an error analysis of the learning algorithm that is established under quite milder assumptions compared to the unregularized case. Therefore, this work covers a wider range of systems in practice.

The paper is set out as follows. In Section~\ref{mfgs}, we introduce classical and regularized mean-field games as well as finite-agent game, and define the classical and regularized mean-field equilibria. In Section~\ref{known-model}, we define mean-field equilibrium operator and show that the mean-field equilibrium operator is contractive. In Section~\ref{unknown-model}, we establish a Q-learning algorithm to compute approximate regularized-mean-field equilibrium and prove its convergence. In Section~\ref{num_ex}, we provide a numerical example to illustrate the effectiveness of the learning algorithm. Section~\ref{conc} concludes the paper.

\noindent\textbf{Notation.} 
For a finite set $\sE$, we let $\Pnew(\sE)$ denote the set of all probability distributions on $\sE$. In this paper, $\|\cdot\|_1$ denotes $l_1$-norm on $\Pnew(\sE)$. Total variation norm on $\Pnew(\sE)$ is denoted by $\|\cdot\|_{TV}$. For any probability measures $\mu,\nu \in \Pnew(\sE)$, we have 
$\|\mu-\nu\|_{TV} = \inf\left\{E^{\xi}[1_{\{x \neq y\}}]: \xi(\cdot,\sE) = \mu(\cdot) \text{ } \text{and} \text{ } \xi(\sE,\cdot) = \nu(\cdot)\right\}$ and the distribution $\xi$ on $\sE \times \sE$ that achieves this infimum is called optimal coupling between $\mu$ and $\nu$. It is known that $\|\cdot\|_1 = 2 \, \|\cdot\|_{TV}$ \cite[p. 141]{Geo11}.
In this paper, we will always endow $\Pnew(\sE)$ with $l_1$-norm. For any $e \in \sE$, $\delta_e$ is the Dirac delta distribution. We let $m(\cdot)$ denote the Lebesgue measure on appropriate finite dimensional Euclidean space $\R^d$. For any $a \in \R^d$ and $\rho > 0$, let $B(a,\rho) \coloneqq \{b: \|a-b\|_1 \leq \rho\}$. For any $a,b \in \R^d$, $\langle a,b \rangle$ denotes the inner product. Let $Q: \sE_1 \times \sE_2 \rightarrow \R$, where $\sE_1$ and $\sE_2$ are two sets. Then, we define $Q_{\max}(e_1) \coloneqq \sup_{e_2 \in \sE_2} Q(e_1,e_2)$. For any function class ${\mathcal G}$, let $V_{\mathcal{G}}$ denote its pseudo-dimension \cite{Vid10}. The notation $v\sim \nu$ means that the random element $v$ has distribution $\nu$. 

\section{Mean-Field Games}\label{mfgs}

A discrete-time mean-field game is specified by
$$
\left( \sX,\sA,p,r \right),
$$
where $\sX$ is the finite state space and $\sA$ is the finite action space. The components $p : \sX \times \sA \times \Pnew(\sX) \to \Pnew(\sX)$ and $r: \sX \times \sA \times \Pnew(\sX) \rightarrow [0,\infty)$ are the transition probability and the one-stage reward function, respectively. Therefore, given current state $x(t)$, action $a(t)$, and state-measure $\mu$, the reward $r(x(t),a(t),\mu)$ is received immediately, and the next state $x(t+1)$ evolves to a new state probabilistically according to the following distribution:
$$
x(t+1) \sim p(\cdot|x(t),a(t),\mu). 
$$
To complete the description of the model dynamics, we should also specify how the agent selects its action. To that end, a policy $\pi$ is a conditional distribution on $\sA$ given $\sX$; that is, $\pi: \sX \rightarrow \Pnew(\sA)$. Let $\Pi$ denote the set of all policies.

In mean-field games, a state-measure $\mu \in \Pnew(\sX)$ represents the collective behavior of the other agents\footnote{In classical mean-field game literature, the exogenous behaviour of the other agents is in general modeled by a state measure-flow $\{\mu_t\}$, $\mu_t \in \Pnew(\sX)$ for all $t$, which means that total population behaviour is non-stationary. In this paper, we only consider the stationary case; that is, $\mu_t = \mu$ for all $t$. Establishing a learning algorithm for the non-stationary case is more challenging.}; that is, $\mu$ can be considered as the infinite population limit of the empirical distribution of the states of other agents.

In this paper, we impose the following assumptions on the system components.

\begin{assumption}
	\label{as1}
	\begin{itemize}
		\item [ ]
		\item [(a)] The one-stage reward function $r$ satisfies the following Lipschitz bound:  
			\begin{multline*}
			|r(x,a,\mu) - r(\hat{x},\hat{a},\hat{\mu})| \\
			\hspace{-20pt}\leq L_1 \, \left( 1_{\{x \neq \hat{x}\}} + 2 \cdot 1_{\{a \neq \hat{a}\}}+\|\mu-\hat{\mu}\|_1 \right), \forall x,\hat{x}, \forall a,\hat{a}, \forall \mu, \hat{\mu}. \nonumber 
			\end{multline*}
		\item [(b)] The stochastic kernel $p(\,\cdot\,|x,a,\mu)$ satisfies the following Lipschitz bound:
		\begin{multline*}
		\|p(\cdot|x,a,\mu) - p(\cdot|\hat{x},\hat{a},\hat{\mu})\|_1 \\
		\hspace{-20pt}\leq K_1 \, \left( 1_{\{x \neq \hat{x}\}} + 2 \, \cdot  1_{\{a \neq \hat{a}\}}+\|\mu-\hat{\mu}\|_1 \right), \forall x,\hat{x}, \forall a,\hat{a}, \forall \mu, \hat{\mu}. 
		\end{multline*}
	\end{itemize}
\end{assumption}

Note that we can equivalently describe the model above as follows. In this equivalent model, we take action space to be the set of probability measures $\sU \coloneqq \Pnew(\sA)$ on the original action space $\sA$. Hence, the new action space $\sU$ is an uncountable, convex, and compact subset of $\R^{\sA}$ with dimension $|\sA|-1$. With this new action space, the new transition probability $P: \sX \times \sU \times \Pnew(\sX) \rightarrow \Pnew(\sX)$ and the new one-stage reward function $R: \sX \times \sU \times \Pnew(\sX) \rightarrow \R$ are defined as follows:
\begin{align}
P(\,\cdot\,|x,u,\mu) &\coloneqq \sum_{a \in \sA} p(\,\cdot\,|x,a,\mu) \, u(a), \nonumber \\
R(x,u,\mu) &\coloneqq \sum_{a \in \sA} r(x,a,\mu) \, u(a). \nonumber 
\end{align}
In this equivalent model, a policy $\pi$ is a deterministic function from state space $\sX$ to the new action space $\sU$. Therefore, for a fixed $\mu$ and $\pi$, the states and actions are evolved as follows: 
\begin{align}
x(t) &\sim P(\,\cdot\,|x(t-1),u(t-1),\mu), \text{ } t\geq1, \nonumber \\
u(t) &=\pi(x(t)), \text{ } t\geq0. \nonumber
\end{align}

In the remainder of this paper, we replace the original mean-field game model with this equivalent one. We prove below the conditions satisfied by the new transition probability $P$ and one-stage reward function $R$ under Assumption~\ref{as1}. 

\begin{proposition}\label{new-con}
	Under Assumption~\ref{as1}, $P$ and $R$ satisfy the following Lipschitz bounds:
	\begin{align}
	|R(x,u,\mu) - R(\hat{x},\hat{u},\hat{\mu})|  &\leq L_1 \, \left( 1_{\{x \neq \hat{x}\}} + \|u-\hat{u}\|_1 +\|\mu-\hat{\mu}\|_1 \right),\nonumber \\
	\|P(\cdot|x,u,\mu) - P(\cdot|\hat{x},\hat{u},\hat{\mu})\|_1 &\leq K_1 \, \left( 1_{\{x \neq \hat{x}\}} + \|u-\hat{u}\|_1 +\|\mu-\hat{\mu}\|_1 \right),\nonumber \end{align}
	$\forall x,\hat{x}, \forall u,\hat{u}, \forall \mu, \hat{\mu}$.
\end{proposition}

\begin{proof}
The proof is in Appendix~\ref{app0}.\qed
\end{proof}

In the following  section, we introduce regularized mean-field games and the adapted optimality notion.


\subsection{Regularized Mean-Field Games}

A theory of regularized Markov decision processes (MDPs) has been introduced in \cite{GeScPi19}. In this work, regularization is introduced via subtracting a strongly convex function from the one-stage reward function. This type of  modifications is in general applied to reinforcement learning algorithms to ensure robust learners with improved exploration. We refer the reader to \cite{GeScPi19} for comprehensive review on a variety of regularized MDPs used in the literature.

Analogous to regularized MDPs, in this section, we introduce regularized mean-field games. To that end, let $\Omega: \sU \rightarrow \R$ be a differentiable $\rho$-strongly convex function with respect to the $l_1$-norm $\|\cdot\|_1$ (see Appendix~\ref{dualitysec} for definition). Let $L_{\reg}$ be the Lipschitz
constant of $\Omega$ on $\sU$, whose existence is guaranteed by strong convexity of $\Omega$. The only difference between classical MFGs and regularized ones is the regularization term in the one-stage reward function. In regularized MFGs, the reward function is given by 
$$
R^{\reg}(x,u,\mu) \coloneqq R(x,u,\mu) - \Omega(u). 
$$
A typical example for $\Omega$ is the negative entropy $\Omega(u) = \sum_{a \in \sA} \ln(u(a)) \, u(a)$. Another similar example is the relative entropy between $u$ and uniform distribution; that is, $\Omega(u) = \sum_{a \in \sA} \ln(u(a)) \, u(a) + \ln(|\sA|)$. In both of these examples, as a result of entropy regularization, agent visits optimal as well as almost optimal actions more often and randomly. This improves the exploration of the algorithm. Moreover, due to strong convexity of $\Omega$, Lipschitz sensitivity of the optimal action on state, state-measure, and other uncertain parameters can be established via Legendre-Fenchel duality. This makes the learning algorithm more robust. This is indeed the main motivation here for introducing the regularization term.

Now, it is time to define the optimality notion that is adapted in this paper. To this end, we first define the regularized discounted cost of any policy given any state measure.  

In regularized MFGs, for a fixed $\mu$, the reward function of any policy $\pi$ is given by
\begin{align}
J^{\reg}_{\mu}(\pi,x) &= E^{\pi}\biggl[ \sum_{t=0}^{\infty} \beta^t R^{\reg}(x(t),u(t),\mu) \biggr], \nonumber 
\end{align}
where $\beta \in (0,1)$ is the discount factor and $x$ is the initial state. For this model, we define the set-valued mapping $\Psi^{\reg} : \Pnew(\sX) \rightarrow 2^{\Pi}$  as follows (here, $2^{\Pi}$ is the collection of all subsets of $\Pi$):
$$\Psi^{\reg}(\mu) = \{\hat{\pi} \in \Pi: J^{\reg}_{\mu}(\hat{\pi},x) = \sup_{\pi} J^{\reg}_{\mu}(\pi,x) \text{ }\text{ for all } \text{ } x \in \sX\}.$$
The set $\Psi^{\reg}(\mu)$ is the set of optimal policies for $\mu$.
Similarly, we define the set-valued mapping $\Lambda^{\reg} : \Pi \to 2^{\Pnew(\sX)}$ as follows: for any $\pi \in \Pi$, the state-measure $\mu_{\pi} \in \Lambda^{\reg}(\pi)$ is an invariant distribution of the transition probability $P(\,\cdot\,|x,\pi(x),\mu_{\pi})$; that is, 
\begin{align}
\mu_{\pi}(\,\cdot\,) = \sum_{x \in \sX} P(\,\cdot\,|x,\pi(x),\mu_{\pi})  \, \mu_{\pi}(x). \nonumber
\end{align}
Under Assumption~\ref{as1} and Proposition~\ref{new-con}, $\Lambda^{\reg}(\pi)$ is always nonempty. This can be established via Kakutani's fixed point theorem (see \cite[Lemma 3]{AnKaSa19-b}). Then, the notion of equilibrium for this regularized game model is defined as follows.

\begin{definition}
	A pair $(\pi_*,\mu_*) \in \Pi \times \Pnew(\sX)$ is a \emph{regularized mean-field equilibrium} if $\pi_* \in \Psi^{\reg}(\mu_*)$ and $\mu_* \in \Lambda^{\reg}(\pi_*)$.
\end{definition}

In this paper, our goal is to develop a Q-learning algorithm for computing an approximate regularized mean-field equilibrium when the model is unknown; that is the transition probability $P$ and the one-stage reward function $R$ are not available to the decision maker. To that end, we define the following.

\begin{definition}\label{approxpol}
		Let $(\pi_*,\mu_*) \in \Pi \times \Pnew(\sX)$ be a \emph{regularized mean-field equilibrium}. A policy $\pi_{\varepsilon} \in \Pi$ is an $\varepsilon$-regularized-mean-field equilibrium policy if
		$$
		\sup_{x \in \sX} \|\pi_{\varepsilon}(x) - \pi_*(x)\|_1 \leq \varepsilon.
		$$
\end{definition}

In the next section, we will first introduce a mean-field equilibrium (MFE) operator, which can be used to compute mean-field equilibrium when the model is known, and prove that this operator is contractive. Then, under model-free setting, we approximate this MFE operator with a random one and establish a learning algorithm. Using this random operator, we obtain $\varepsilon$-regularized-mean-field equilibrium policy with high confidence. This learned approximate regularized-mean-field equilibrium policy can then be used in finite-agent game model as an approximate Nash equilibrium.

\section{Mean-Field Equilibrium Operator}\label{known-model}

In this section, we introduce a mean-field equilibrium (MFE) operator, whose fixed point is a mean-field equilibrium. We prove that this operator is contractive. Using this result, we then establish a Q-learning algorithm to obtain approximate regularized mean-field equilibrium policy. To that end, in addition to Assumption~\ref{as1}, we assume the following. This assumption ensures that the MFE operator is contractive. 

\begin{assumption}\label{as2}
		We assume that 
		$$\frac{3 \, K_1}{2} \, \left( 1 + \frac{1}{\rho} \, \frac{K_{\Lip}}{1-\beta}\right) < 1,$$
		where 
		$$K_{\Lip} \coloneqq \frac{L_1}{1-\beta\, K_1/2} > 0.$$
\end{assumption}

Recall that given any state-measure $\mu$, the regularized value function $J^{\reg}_{\mu}$ of policy $\pi$ with initial state $x$ is defined as
$$
J^{\reg}_{\mu}(\pi,x) = E^{\pi}\biggl[ \sum_{t=0}^{\infty} \beta^t R^{\reg}(x(t),u(t),\mu) \, \bigg| \, x(0) = x \biggr]. 
$$
Then, the optimal regularized value function is given by 
$$J^{\reg,*}_{\mu}(x) \coloneqq \sup_{\pi \in \Pi} J^{\reg}_{\mu}(\pi,x).$$ 
Similarly, we define the optimal regularized $Q$-function as
$$
Q^{\reg,*}_{\mu}(x,u) = R^{\reg}(x,u,\mu) + \beta \sum_{y \in \sX} J^{\reg,*}_{\mu}(y) \, P(y|x,u,\mu). 
$$
Note that $Q^{\reg,*}_{\mu,\max}(x) \coloneqq \sup_{u \in \sU} Q^{\reg,*}_{\mu}(x,u) = J_{\mu}^{\reg,*}(x)$ for all $x \in \sX$. Therefore, we have the following optimality equation:
\begin{align}
Q^{\reg,*}_{\mu}(x,u) &= R^{\reg}(x,u,\mu) + \beta \sum_{y \in \sX} Q^{\reg,*}_{\mu,\max}(y) \, P(y|x,u,\mu)  \nonumber \\
&\eqqcolon L_{\mu}Q^{\reg,*}_{\mu}(x,u). \nonumber 
\end{align}
It is also a well-known fact that $Q^{\reg,*}_{\mu,\max}$ satisfies the following Bellman optimality equation:
	\begin{align}
	Q^{\reg,*}_{\mu,\max}(x) 
	&= \sup_{u} \left[ R^{\reg}(x,u,\mu) + \beta \sum_{y \in \sX} Q^{\reg,*}_{\mu,\max}(y) \, P(y|x,u,\mu) \right]  \nonumber \\
	&\eqqcolon T_{\mu}Q^{\reg,*}_{\mu,\max}(x). \nonumber
	\end{align}
Here, $L_{\mu}$ and $T_{\mu}$ are $\|\cdot\|_{\infty}$-contractions with contraction factor $\beta$, and the unique fixed point of $L_{\mu}$ is $Q^{\reg,*}_{\mu}$ and the unique fixed point of $T_{\mu}$ is $Q^{\reg,*}_{\mu,\max}$. 

Let ${\mathcal C}$ denote the set of all $Q$-functions satisfying the following properties: any $Q \in {\mathcal C}$ is uniformly $\left(K_{\Lip}+L_{\reg}\right)$- Lipschitz continuous and $\rho$-strongly concave with respect to $u$. We endow ${\mathcal C}$ with the sup-norm $\|\cdot\|_{\infty}$ throughout the paper. 

\begin{lemma}\label{lip-value}
		For any $\mu$, $Q^{\reg,*}_{\mu,\max}$ is $K_{\Lip}$-Lipschitz continuous; that is, 
		$$
		|Q^{\reg,*}_{\mu,\max}(x)-Q^{\reg,*}_{\mu,\max}(y)| \leq K_{\Lip} \, 1_{\{x \neq y\}}.
		$$
\end{lemma}
	
\begin{proof}
The proof is in Appendix~\ref{app01}.\qed		 
\end{proof}

Now, we define the MFE operator. To that end, we define $H_1: \Pnew(\sX) \rightarrow {\mathcal C}$ as $H_1(\mu) = Q^{\reg,*}_{\mu}$ (optimal regularized Q-function) and $H_2: \Pnew(\sX) \times {\mathcal C} \rightarrow \Pnew(\sX)$ as
$$
H_2(\mu,Q)(\cdot) \coloneqq \sum_{x \in \sX} P(\cdot|x,f_{Q}(x),\mu) \, \mu(x),
$$
where $f_{Q}(x) = \argmax_{u\in\sU} Q(x,u)$ for all $x \in \sX$. With these definitions, we can give the definition of the optimality operator as follows:
$$H: \Pnew(\sX) \ni \mu \mapsto H_2\left(\mu,H_1(\mu)\right) \in \Pnew(\sX).$$ 
Our goal is to prove that $H$ is contractive. In the following lemma, we prove that $H_1$ is Lipschitz, which will be used to prove that operator $H$ is contractive.  

\begin{lemma}\label{n-lemma1}
	The mapping $H_1$ is a Lipschitz continuous with the Lipschitz constant $\displaystyle K_{H_1} \coloneqq \frac{K_{\Lip}}{1-\beta}$.
\end{lemma}
	
\proof
The proof is in Appendix~\ref{app02}\qed
\endproof

Before we prove that $H$ is contractive, we establish that for any mean-field term, the optimal policy is Lipschitz continuous with respect to the mean-field term.

\begin{lemma}\label{Lips-pol}
	For any $\mu,\hmu$, let $f_{\mu}$ and $f_{\hmu}$ denote the corresponding optimal policies. Then, it follows that
	\begin{align}
		\|f_{\mu}(x)-f_{\hmu}(\hx)\|_1 \leq \frac{1}{\rho} \, K_{H_1} (1_{\{x \neq \hx\}} + \|\mu-\hmu\|_1), \nonumber
	\end{align}
	for all $x,\hx$. 
\end{lemma}

\begin{proof}
The proof is in Appendix~\ref{app03}	\qed
\end{proof}

\begin{remark}
In the absence of the regularization term, one can establish the Lipschitz continuity of the optimal policy with respect to the mean-field term if it is assumed that the following function
$$
F: (x,v,\mu,u) \mapsto R(x,u,\mu) + \beta \sum_y v(y) \, P(y|x,u,\mu)
$$
is strongly concave with respect to $u$ and has a Lipschitz continuous gradient in $u$ with respect to $x,v,\mu$, which are in general restrictive conditions. Indeed, these conditions were imposed in our previous work \cite[Assumption 2.1(d)]{AnKaSa19-b} on learning unregularized mean-field games. As a result of these restrictive conditions, the analysis of the convergence of the algorithm is much more involved. Therefore, introducing a regularization term into the one-stage reward function significantly relaxes these conditions on the system components and simplifies the analysis. Moreover, because of
the Lipschitz sensitivity of the optimal policy, the algorithm is supposed to be more robust to the uncertainties in the environment.  
	
Note that in classical algorithms developed for MDPs, such as $Q$-learning, value iteration, and policy iteration, it is not required to establish the Lipschitz continuity of the optimal policy. However, in mean-field games, since the optimal policy $f_{\mu}$ directly affects the behaviour of the next mean-field term through
$$
	H_2(\mu,Q_{\mu}^{\reg,*})(\cdot) = \sum_x P(\cdot|x,f_{\mu}(x),\mu) \, \mu(x),
$$
one must also establish the Lipschitz continuity of the optimal policy $f_{\mu}$ in mean-field games. This is indeed the most challenging part in the analysis compared to the analysis of the algorithms developed for MDPs.
\end{remark}

Now, we can prove using Lemma~\ref{n-lemma1} and Lemma~\ref{Lips-pol}, that $H$ is contractive. 

\begin{proposition}\label{MFE-con}
	The mapping $H$ is a contraction with the contraction constant $K_{H}$, where
	$$
	K_H \coloneqq \frac{3 \, K_1}{2} \, \left( 1 + \frac{ K_{H_1}}{\rho} \right). 
	$$
\end{proposition}

\proof
Now, fix any $\mu, \hmu \in \Pnew(\sX)$. Using Lemma~\ref{Lips-pol}, we have
\begin{align}
&\|H_2(\mu,H_1(\mu)) - H_2(\hmu,H_1(\hmu))\|_1 \nonumber \\
&= \sum_{y} \, \bigg| \sum_{x} \, P(y|x,f_{\mu}(x),\mu) \, \mu(x) - \sum_{x} \, P(y|x,f_{\hmu}(x),\hmu) \, \hmu(x) \biggr| \nonumber \\
&\leq \sum_{y} \, \bigg| \sum_{x} \, P(y|x,f_{\mu}(x),\mu) \, \mu(x) - \sum_{x} \, P(y|x,f_{\hmu}(x),\hmu) \, \mu(x) \biggr| \nonumber \\
&\phantom{xx}+ \sum_{y} \, \bigg| \sum_{x} \,P(y|x,f_{\hmu}(x),\hmu) \, \mu(x) - \sum_{x} \, P(y|x,f_{\hmu}(x),\hmu) \, \hmu(x) \biggr| \nonumber \\
&\overset{(I)}{\leq} \sum_{x} \left\|P(\cdot|x,f_{\mu}(x),\mu)-P(\cdot|x,f_{\hmu}(x),\hmu) \right\|_1 \mu(x) + \frac{K_1}{2} \left( 1 + \frac{K_{H_1}}{\rho} \right) \, \|\mu-\hmu\|_1 \nonumber \\
&\leq K_1 \left( \sup_x \|f_{\mu}(x)-f_{\hmu}(x)\|_1 + \|\mu-\hmu\|_1 \right) + \frac{K_1}{2} \left( 1 + \frac{K_{H_1}}{\rho} \right) \, \|\mu-\hmu\|_1 \nonumber \\
&\leq  \frac{3 \, K_1}{2} \, \left( 1 + \frac{K_{H_1}}{\rho} \right) \, \|\mu-\hmu\|_1.
\end{align}
Note that Lemma~\ref{Lips-pol} and Proposition~\ref{new-con} lead to
\begin{equation*}
\|P(\cdot|x,f_{\hmu}(x),\hmu)-P(\cdot|y,f_{\hmu}(y),\hmu)\|_1 \leq K_1 \left( 1 + \frac{K_{H_1}}{\rho} \right) \, 1_{\{x \neq y\}}.
\end{equation*}
Hence, (I) follows from \cite[Lemma A2]{KoRa08}. This completes the proof. \qed
\endproof

Under Assumption~\ref{as1} and Assumption~\ref{as2}, $H$ is a contraction mapping. Therefore, by Banach Fixed Point Theorem, $H$ has an unique fixed point. Let $\mu_*$ be this unique fixed point and $Q_{\mu_*}^{\reg,*} = H_1(\mu_*)$. Let $\pi_*(x) = f_{\mu_*}(x)$. Then, one can prove that the pair $(\pi_*,\mu_*)$ is a regularized mean-field equilibrium. Hence, we can compute this regularized mean-field equilibrium via applying $H$ recursively starting from arbitrary $\mu_0$. This indeed leads to a value iteration algorithm for computing mean-field equilibrium. However, if the model is unknown; that is the transition probability $P$ and the one-stage reward function $R$ are not available to the decision maker, we replace $H$ with a random operator and establish a learning algorithm via this random operator. To prove the convergence of this learning algorithm, the contraction property of $H$ is crucial.

\section{Finite Agent Game}\label{sec2}

The regularized mean-field game model in Section~\ref{mfgs} is indeed the infinite-population limit of the regularized finite-agent game model that will be described below. In a finite-agent game model, we have $N$-agents and, for each agent $i \in \{1,2,\ldots,N\}$, $x^N_i(t) \in \sX$ and $u^N_i(t) \in \sU$ denote the state and the action of Agent~$i$ at time $t$, respectively. 
The empirical distribution of the states of agents at time $t$ is defined as follows: 
\begin{align}
e_t^{(N)}(\,\cdot\,) \coloneqq \frac{1}{N} \sum_{i=1}^N \delta_{x_i^N(t)}(\,\cdot\,) \in \Pnew(\sX). \nonumber
\end{align}
This empirical distribution affects both the system dynamics and one-stage reward function. Therefore, for each $t \ge 0$, next states $(x^N_1(t+1),\ldots,x^N_N(t+1))$ of agents have the following conditional distribution given current states $(x^N_1(t),\ldots,x^N_N(t))$ and actions $(u^N_1(t),\ldots,u^N_N(t))$:
\begin{align}
&\prod^N_{i=1} P\big(x^N_i(t+1)\big|x^N_i(t),u^N_i(t),e^{(N)}_t\big). \nonumber 
\end{align}
A \emph{policy} $\pi$ for a generic agent is a deterministic function from $\sX$ to $\sU$. The set of all policies for Agent~$i$ is denoted by $\Pi_i$. The initial states $x^N_i(0)$ are independent and identically distributed according to $\mu_0$.

Let ${\boldsymbol \pi}^{(N)} \coloneqq (\pi^1,\ldots,\pi^N)$, $\pi^i \in \Pi_i$, denote an $N$-tuple of policies. Under such an $N$-tuple of policies, the regularized discounted reward of Agent~$i$ is defined as
\begin{align}
J_i^{(N)}({\boldsymbol \pi}^{(N)}) &= E^{{\boldsymbol \pi}^{(N)}}\biggl[\sum_{t=0}^{\infty}\beta^{t}R^{\reg}(x_{i}^N(t),u_{i}^N(t),e^{(N)}_t)\biggr]. \nonumber 
\end{align}
Then, the goal of the agents is to achieve a Nash equilibrium, which is defined as follows.

\begin{definition}
	An $N$-tuple of policies ${\boldsymbol \pi}^{(N*)}= (\pi^{1*},\ldots,\pi^{N*})$ is a \emph{Nash equilibrium} if
	\begin{align}
	J_i^{(N)}({\boldsymbol \pi}^{(N*)}) = \sup_{\pi^i \in \Pi_i} J_i^{(N)}({\boldsymbol \pi}^{(N*)}_{-i},\pi^i) \nonumber
	\end{align}
	for each $i=1,\ldots,N$, where ${\boldsymbol \pi}^{(N*)}_{-i} \coloneqq (\pi^{j*})_{j\neq i}$.
\end{definition}

It is known that establishing the existence of Nash equilibria and computing it are in general prohibitive for finite-agent game model as a result of the decentralized nature of the problem (see \cite[pp. 4259]{SaBaRaSIAM}). Therefore, it is of interest to obtain an approximate Nash equilibrium, whose definition is given below.

\begin{definition}
	An $N$-tuple of policies ${\boldsymbol \pi}^{(N*)}= (\pi^{1*},\ldots,\pi^{N*})$ constitutes an \emph{$\delta$-Nash equilibrium} if
	\begin{align}
	J_i^{(N)}({\boldsymbol \pi}^{(N*)}) \geq \sup_{\pi^i \in \Pi_i} J_i^{(N)}({\boldsymbol \pi}^{(N*)}_{-i},\pi^i) - \delta \nonumber
	\end{align}
	for each $i=1,\ldots,N$, where ${\boldsymbol \pi}^{(N*)}_{-i} \coloneqq (\pi^{j*})_{j\neq i}$.
\end{definition}

Due to symmetry in mean-field game model, if the number of agents is large enough, one can obtain approximate Nash equilibrium by studying the infinite population limit $N\rightarrow\infty$ of the game (i.e., mean-field game model in Section~\ref{mfgs}). Indeed, one can prove that if each agent in the finite-agent game model adopts the $\varepsilon$-regularized-mean-field equilibrium policy in Definition~\ref{approxpol} of the infinite population limit, the resulting policy will be an approximate Nash equilibrium for all sufficiently large $N$-agent game models. Indeed, this is the statement of the below theorem.

Before we state the theorem, let us define the following constants:
\small
\begin{align}
	&C_1 \coloneqq \left(\frac{3 \, K_1}{2}  + \frac{K_1 \, K_{H_1}}{2\rho} \right), \, C_2 \coloneqq \left(L_1+\frac{\beta K_1 K_{\Lip}}{2}\right) \frac{K_1}{1-C_1} \nonumber \\
	&C_3 \coloneqq \left(L_1+L_{\reg}+\frac{\beta K_1 K_{\Lip}}{2}\right).\nonumber
\end{align}
\normalsize
Note that by Assumption~\ref{as2}, the constant $C_1$ is strictly less than $1$.
	
\begin{theorem}\label{old-main-cor}
	Let $\pi_{\varepsilon}$ be an $\varepsilon$-regularized-mean-field equilibrium policy for the mean-field equilibrium $(\pi_*,\mu_*)$. Let $\mu_0 \in \Lambda^{\reg}(\pi_{\varepsilon})$. Then, for any $\delta>0$, there exists a positive integer $N(\delta)$, such that, for each $N\geq N(\delta)$, the $N$-tuple of policies ${\boldsymbol \pi}^{(N)} = \{\pi_{\varepsilon},\pi_{\varepsilon},\ldots,\pi_{\varepsilon}\}$ is an $(\tau \, \varepsilon + \delta)$-Nash equilibrium for the game with $N$ agents, where 
	$
		\tau \coloneqq \frac{2C_2+C_3}{1-\beta}.
	$
\end{theorem}
	
\proof
The proof is in Appendix~\ref{app1}.\qed
\endproof


In the next section, we develop an algorithm for learning $\varepsilon$-regularized-mean-field equilibrium policy via fitted Q-iteration and empirical estimation of the transition probability. Therefore, if each agent in the finite-agent game model adopts this learned policy, then the resulting policy will be an approximate Nash equilibrium for finite-agent setup.

\section{Q-Learning Algorithm}\label{unknown-model}

In this section, we establish an offline learning algorithm  for obtaining approximate regularized mean-field equilibrium. We suppose that a generic agent has access to a simulator, which generates a new state $y \sim P(\,\cdot\,|x,u,\mu)$ and gives the reward $R(x,u,\mu)$ for any given state $x$, action $u$, and state measure $\mu$. This is a typical assumption in offline reinforcement learning algorithms.

In this learning algorithm, we replace operators $H_1$ and $H_2$ with random operators $\hat{H}_1$ and $\hat{H}_2$, respectively. Therefore, we have two stages in each iteration of the learning algorithm. In the first stage, the optimal regularized Q-function $Q_{\mu}^{\reg,*}$ for a given $\mu$ is learned via fitted Q-learning algorithm, which has been introduced in \cite{AnMuSz07} to learn optimal Q-functions of Markov decision processes. This stage replaces the operator $H_1$ with a random operator $\hat{H}_1$. In this fitted Q-learning algorithm, Q-functions are picked from a fixed function class ${\mathcal F}$. This function class ${\mathcal F}$ can be chosen as the set of neural networks with some fixed architecture or linear span of some finite number of basis functions or the set ${\mathcal C}$ itself. Depending on ${\mathcal F}$, an additional representation error in the learning algorithm will be present. Let ${\mathcal F}_{\max} \coloneqq \{Q_{\max}: Q \in {\mathcal F}\}$. 

In the second stage of each iteration, the state-measure is updated via simulating corresponding transition probability. This stage replaces the operator $H_2$ with a random operator $\hat{H}_2$.

\begin{remark}
Note that this learning algorithm can be applied to finite-agent game problem as follows. First of all, we must assume that each agent has access to a simulator, which generates a new state $y \sim P(\,\cdot\,|x,u,\mu)$ and gives the reward $R(x,u,\mu)$ for any given state $x$, action $u$, and state measure $\mu$. This is a typical assumption in offline reinforcement learning algorithms. Using this simulator, each agent runs the proposed learning algorithm offline to compute an approximate regularized mean-field equilibrium policy. Agents then have to agree on learned approximate regularized mean-field equilibrium policies via running some consensus algorithm. Then the resulting joint policy will be approximate Nash equilibrium by Theorem~\ref{old-main-cor}. \end{remark}

We now proceed by giving the description of $\hat{H}_1$ first. Let $\nu$ be a probability measure on $\sX$ such that $\min_{x \in \sX} \nu(x) > 0$. Define $\xi_0 \coloneqq 1/\sqrt{\min_{x \in \sX} \nu(x) }$. We fix some function $\pi_b: \sX \rightarrow \Pnew(\sU)$ such that, for any $x \in \sX$, the distribution $\pi_b(x)(\cdot) $ on $\sU$ has a density with respect to Lebesgue measure $m$. With an abuse of notation, we denote this density with $\pi_b(x,u)$. We assume that $\pi_0 \coloneqq \inf_{(x,u) \in \sX\times\sU} \pi_b(x,u) > 0$. Now, we can give the definition of the random operator $\hat{H}_1$.

\begin{algorithm}[h]
	\caption{Algorithm $\hat{H}_1$}
	\label{H1}
	\begin{algorithmic}
		\STATE{Inputs $\left([N,L],\mu\right)$}
		\STATE{generate i.i.d. samples $\{(x_t,u_t,r_t,y_{t+1})_{t=1}^N\}$ using
			\begin{displaymath}
			x_t \sim \nu, \, u_t \sim \pi_b(x_t)(\cdot), \, r_t = R^{\reg}(x_t,u_t,\mu), \, y_{t+1} \sim P(\cdot|x_t,u_t,\mu) 
			\end{displaymath}
			}
		\STATE{Start with $Q_0 = 0$}
		\FOR{$l=0,\ldots,L-1$} 
        \STATE{  
			\begin{equation*}
			Q_{l+1} = \argmin_{f \in {\mathcal F}} \frac{1}{N} \sum_{t=1}^N \frac{1}{m(\sU) \, \pi_b(x_t,u_t)} \bigg| f(x_t,u_t) - \left[r_t + \beta \max_{u' \in \sU} Q_l(y_{t+1},u') \right]\bigg|^2 
			\end{equation*}
			}
	    \ENDFOR
		\RETURN{$Q_L$}
	\end{algorithmic}
\end{algorithm}

\begin{remark}
	Note that in Algorithm 1, one can alternatively use sample path $\{x_t,u_t\}_{t=1}^N$ generated by the policy $\pi_b$ instead of using i.i.d. samples. In this case, under $\pi_b$, the state process $\{x_t\}$ should assumed to be strictly stationary and exponentially $\beta$-mixing \cite{AnMuSz07}. Since exponentially $\beta$-mixing stationary processes forget its past exponentially fast, when there is a sufficiently large time difference between two samples, they behave like i.i.d. processes. Therefore, this makes the error analysis of the exponential $\beta$-mixing case almost the same as the i.i.d. case. However, the main problem in $\beta$-mixing case is finding a policy $\pi_b$ satisfying this mixing condition. We refer the reader to \cite{AnMuSz07-t,AnMuSz07} for the details of the error analysis of $\hat{H}_1$ in exponentially $\beta$-mixing case.  
\end{remark}

Before we describe $\hat{H}_2$, the error analysis of algorithm $\hat{H}_1$ is given. Note that there exists $\alpha > 0$ such that for any $u \in \sU$ and $\xi >0$, we have 
$
m\left( B(u,\xi) \cap \sU \right) \geq \min\left\{ \alpha \, m(B(u,\xi)), m(\sU) \right\},
$ 
where $m$ is the Lebesgue measure on $\sU$ when considered as a subset of $\R^{|\sA|-1}$ (see \cite{AnMuSz07-t}). For any $Q \in \C$, we define $v$-norm of $Q$ as
\begin{align}
	\|Q\|_v \coloneqq \left[\sum_{x} \int_{\sU} |Q(x,u)|^2 \, \frac{m(du)}{m(\sU)} \, v(x)\right]^{1/2}. \nonumber 
\end{align}
We also define $r_{\mathbf{m}} \coloneqq \sup_{(x,u,\mu) \in \sX \times \sU \times \Pnew(\sX)} |R^{\reg}(x,u,\mu)|$ and $Q_{\mathbf{m}} \coloneqq r_{\mathbf{m}}/(1-\beta)$. Using these, we need to define the following constants:
\begin{align*}
	&E({\mathcal F}) \coloneqq \sup_{\mu \in \Pnew(\sX)} \sup_{Q \in {\mathcal F}} \inf_{Q' \in {\mathcal F}} \|Q'-H_{\mu}Q\|_v  \\
	&L_{\mathbf{m}} \coloneqq (1+\beta) Q_{\mathbf{m}} + r_{\mathbf{m}}, \,\,\, C \coloneqq \frac{L_{\mathbf{m}}^2}{m(\sU) \, \pi_0}  \\
	&\Upsilon = 8 \, e^2 \, (V_{{\mathcal F}}+1) \, (V_{{\mathcal F}_{\max}}+1) \left(\frac{64 e Q_{\mathbf{m}} L_{\mathbf{m}} (1+\beta)}{m(\sU) \pi_0}\right)^{V_{{\mathcal F}}+V_{{\mathcal F}_{\max}}}  \\
	&V = V_{{\mathcal F}}+V_{{\mathcal F}_{\max}}, \,\,\, \gamma = 512 C^2  \\
	&\Delta \coloneqq \frac{1}{1-\beta} \left[ \frac{m(\sU) |\sA|! \xi_0}{\alpha (2/(K_{\Lip}+L_{\reg}))^{|\sA|-1}} \, E({\mathcal F}) \right]^{\frac{1}{|\sA|}}  \\
	&\Lambda \coloneqq \frac{1}{1-\beta} \left[ \frac{ m(\sU) |\sA|! \xi_0}{\alpha (2/(K_{\Lip}+L_{\reg}))^{|\sA|-1}} \right]^{\frac{1}{|\sA|}}.  
\end{align*}

Here $E({\cal F})$ gives the representation error of the function class ${\cal F}$. This error in general is zero or very small, since any $Q$ function in ${\cal C}$ can be approximated very well via, for instance, neural networks with some fixed architecture. Hence, we can think of the error due to $E({\cal F})$ negligible. 
The following theorem gives the error analysis of the algorithm $\hat{H}_1$.

\begin{theorem}{(\cite[Theorem 4.1]{AnKaSa19-b})}\label{Thm-H1}
	For any $(\varepsilon,\delta) \in (0,1)^2$, with probability at least $1-\delta$, we have 
	$$
	\left\|\hat{H}_1[N,L](\mu) - H_1(\mu)\right\|_{\infty} \leq \varepsilon + \Delta
	$$
	if $\frac{\beta^L}{1-\beta} \, Q_{\mathbf{m}} < \frac{\varepsilon}{2}$ and $N \geq m_1(\epsilon,\delta,L)$, where 
	$$
	m_1(\varepsilon,\delta,L) \coloneqq \frac{\gamma (2 \Lambda)^{4|\sA|}}{\varepsilon^{4|\sA|}} \, \ln\left(\frac{\Upsilon (2 \Lambda)^{2V|\sA|} L}{\delta \varepsilon^{2V|\sA|}}\right). 
	$$
	Here, the constant error $\Delta$ is as a result of the representation error $E({\mathcal F})$ in the algorithm, which is in general negligible.  
\end{theorem}

Next, we describe the random operator $\hat{H}_2$, and then, give the error analysis.   

\begin{algorithm}[h]
	\caption{Algorithm $\hat{H}_2$}
	\label{H2}
	\begin{algorithmic}
		\STATE{Inputs $\left(M,\mu,Q\right)$}
		\FOR{$x \in \sX$}
		\STATE{generate i.i.d. samples $\{y_t^x\}_{t=1}^M$ using
			$$
			y_t^x \sim P(\cdot|x,f_{Q}(x),\mu)
			$$
			and define 
			$$
			P_M(\cdot|x,f_{Q}(x),\mu) = \frac{1}{M} \sum_{t=1}^M \delta_{y_t^x}(\cdot).
			$$
		}
		\ENDFOR
		\RETURN{$\sum_{x \in \sX} P_M(\cdot|x,f_{Q}(x),\mu) \, \mu(x)$}
	\end{algorithmic}
\end{algorithm}

\begin{theorem}{(\cite[Theorem 4.2]{AnKaSa19-b})}\label{Thm-H2}
	For any $(\varepsilon,\delta) \in (0,1)^2$, with probability at least $1-\delta$
	$$
	\left\|\hat{H}_2[M](\mu,Q) - H_2(\mu,Q) \right\|_1 \leq \varepsilon 
	$$
	if $M \geq m_2(\epsilon,\delta)$, where 
	$$
	m_2(\epsilon,\delta) \coloneqq \frac{|\sX|^2}{\varepsilon^2} \, \ln\left(\frac{2 \, |\sX|^2}{\delta}\right). 
	$$
\end{theorem}

Algorithm 3 provides the overall description of the algorithm $\hat{H}$, which replaces the MFE operator $H$.     

\begin{algorithm}[H]
	\caption{Algorithm $\hat{H}$}
	\label{Qit}
	\begin{algorithmic}
		\STATE{Inputs $\left(K,\{[N_k,L_k]\}_{k=0}^{K},\{M_k\}_{k=0}^{K-1},\mu_0\right)$}
		\STATE{Start with $\mu_0$}
		\FOR{$k=0,\ldots,K-1$}
		\STATE{
			\begin{align}
			\mu_{k+1} &= \hat{H}\left([N_k,L_k],M_k\right)(\mu_k) \nonumber \\
			&\coloneqq \hat{H}_2[M_k]\left(\mu_k,\hat{H}_1[N_k,L_k](\mu_k)\right) \nonumber 
			\end{align}
		}
		\ENDFOR
		\RETURN{$\mu_K$}
	\end{algorithmic}
\end{algorithm}

Note that in Algorithm 3, for each stage $k=0,\ldots,K-1$, the input is $\mu_k$. In addition, we also pick integers $N_k$ and $L_k$ as inputs for the random operator $\hat{H_1}$ and pick integer $M_k$ as an input for the random operator $\hat{H}_2$ at each stage. We first compute an approximate $Q$-function for $\mu_k$ via $\hat{H}_1[N_k,L_k](\mu_k)$ and we compute an approximate new mean-field term via $\hat{H}_2[M_k]\left(\mu_k,\hat{H}_1[N_k,L_k](\mu_k)\right)$. In the second stage of the iteration, since we are using an approximate $Q$-function instead of the exact $Q$-function, we also have an error due to $\hat{H}_1$ in addition to the error resulting from $\hat{H_2}$.

Using above error analyses of the algorithms $\hat{H_1}$ and $\hat{H}_2$, we can now obtain the following error analysis for the algorithm $\hat{H}$. Then, the main result of this paper can be stated as a corollary of this result. 

\begin{theorem}\label{main-result}
	Fix any $(\varepsilon,\delta) \in (0,1)^2$. Define 
	$$
	\varepsilon_1 \coloneqq \frac{(1-K_H)^2 \, \varepsilon^2}{16 \, \theta \, (K_1)^2}, \,\,
	\varepsilon_2 \coloneqq \frac{(1-K_H) \, \varepsilon}{4},
	$$
	where $\theta \coloneqq \frac{4}{\rho}$. Let $K,L$ be such that 
	\begin{align}
	\frac{(K_H)^K}{1-K_H} &\leq \frac{\varepsilon}{2}, \,\, \frac{\beta^L}{1-\beta} Q_{\mathbf{m}} \leq \frac{\varepsilon_1}{2}. \nonumber 
	\end{align}
	Then, pick $N,M$ such that
	\begin{align}
	N &\geq m_1\left( \varepsilon_1,\frac{\delta}{2K},L \right), \,\, M \geq m_2\left( \varepsilon_2,\frac{\delta}{2K} \right).
	\end{align}
	Let $\mu_K$ be the output of the learning algorithm established by random operator $\hat{H}$ with inputs $$\left(K, \{[N,L]\}_{k=0}^K, \{M\}_{k=0}^{K-1}, \mu_0 \right).$$ Then, with probability at least $1-\delta$
	$$
	\|\mu_K - \mu_*\|_1 \leq \frac{K_1 \sqrt{\theta\, \Delta }}{(1-K_H)} + \varepsilon,
	$$
	where $\mu_*$ is the unique fixed point of $H$ in regularized mean-field equilibrium. 
	
\end{theorem}

\proof

Note that for any $\mu \in \Pnew(\sX)$, $Q \in {\cal C}$ ,and $\hQ \in {\mathcal F}$, we have
\small
\begin{align}
\|H_2(\mu,Q) - H_2(\mu,\hQ)\|_1 &= \sum_{y \in \sX} \left| \sum_{x \in \sX} P(y|x,f_{Q}(x),\mu) \, \mu(x) - \sum_{x \in \sX} P(y|x,f_{\hQ}(x),\mu) \, \mu(x) \right| \nonumber \\
&\leq \sum_{x \in \sX} \|P(\cdot|x,f_{Q}(x),\mu)-P(\cdot|x,f_{\hQ}(x),\mu)\|_1 \, \mu(x) \nonumber \\
&\leq \sum_{x \in \sX} K_1 \, \|f_{Q}(x)-f_{\hQ}(x)\|_1 \, \mu(x). \label{fbound}
\end{align}
\normalsize
Suppose that $Q$ is of the following form:
\begin{align}
Q(x,u) &= R^{\reg}(x,u,\mu) + \beta \sum_{y \in \sX} v(y) \, P(y|x,u,\mu) \nonumber \\
&= \langle q_x^{\mu,v}, u \rangle - \Omega(u), \nonumber 
\end{align}
where $v:\sX\rightarrow\R$ and 
$$
q_x^{\mu,v}(\cdot) \coloneqq r(x,\cdot,\mu) + \beta \sum_{y \in \sX} v(y) \, p(y|x,\cdot,\mu).
$$

Note that the mapping $f_{Q}(x)$ is the unique maximizer of $Q(x,\cdot)$ and $f_{\hQ}(x)$ is the  maximizer of $\hQ(x,\cdot)$. Let us set $f_{Q}(x)=u$ and $f_{\hQ}(x) = u'$. Then it follows that
\begin{align}
	Q(x,u) - Q(x,u') 
	&= \langle q_x^{\mu,v}, u \rangle - \Omega(u) - \langle q_x^{\mu,v}, u' \rangle + \Omega(u') \nonumber \\
	&= \langle q_x^{\mu,v}, u-u' \rangle +\Omega(u') - \Omega(u) \nonumber \\
	&\overset{(I)}{\geq}  \langle q_x^{\mu,v}, u-u' \rangle + \langle  \nabla \Omega(u), u'-u \rangle + \frac{\rho}{2} \, \|u-u'\|_1^2 \nonumber \\
	&= \langle \nabla Q(x,u), u-u'\rangle + \frac{\rho}{2}  \, \|u-u'\|_1^2 \nonumber \\
	&\overset{(II)}{\geq}  \frac{\rho}{2}  \, \|u-u'\|_1^2, \nonumber 
\end{align}
where (I) follows from strong convexity of $\Omega$ with respect to $l_1$-norm and (II) follows from first-order optimality condition for differentiable concave functions. 
Now, we have 
\begin{align}
	\|f_{Q}(x) &- f_{\hQ}(x)\|_1^2 \nonumber\\
	&\leq \frac{2}{\rho}  \, \left(Q(x,f_{Q}(x)) - Q(x,f_{\hQ}(x))\right) \nonumber \\
	&= \frac{2}{\rho}  \bigg(Q(x,f_{Q}(x)) - \hQ(x,f_{\hQ}(x)) +\hQ(x,f_{\hQ}(x)) - Q(x,f_{\hQ}(x))\bigg) \nonumber \\
	&= \frac{2}{\rho}  \bigg(\max_{u \in \sU} Q(x,u) - \max_{u \in \sU} \hQ(x,u) + \hQ(x,f_{\hQ}(x)) - Q(x,f_{\hQ}(x)) \bigg) \nonumber \\
	&\leq \frac{4}{\rho} \, \|Q-\hQ\|_{\infty} \eqqcolon \theta \, \|Q-\hQ\|_{\infty}. \label{perturbation2}
\end{align}

Hence, combining (\ref{fbound}) and (\ref{perturbation2}) yields
\begin{align}
\|H_2(\mu,Q) - H_2(\mu,\hQ)\|_1 \leq \sqrt{\theta} \, K_1 \, \sqrt{\|Q-\hQ\|_{\infty}}. \label{mainbound}
\end{align}

Using (\ref{mainbound}), for any $k=0,\ldots,K-1$, we have
\begin{align}
\|H(\mu_k) &- \hat{H}([N,L],M)(\mu_k)\|_1 \nonumber \\
&\leq \|H_2(\mu_k,H_1(\mu_k)) - H_2(\mu_k,\hat{H}_1[N,L](\mu_k))\|_1 \nonumber \\
&\phantom{xx} + \|H_2(\mu_k,\hat{H}_1[N,L](\mu_k)) - \hat{H}_2[M](\mu_k,\hat{H}_1[N,L](\mu_k))\|_1  \nonumber \\
&\leq \sqrt{\theta}  \, K_1 \, \sqrt{\|H_1(\mu_k) - \hat{H}_1[N,L](\mu_k)\|_{\infty}} \nonumber \\
&\phantom{xx} + \|H_2(\mu_k,\hat{H}_1[N,L](\mu_k)) - \hat{H}_2[M](\mu_k,\hat{H}_1[N,L](\mu_k))\|_1. \nonumber 
\end{align}
The last term is bounded from above by 
$$K_1 \sqrt{\theta (\varepsilon_1 + \Delta)} + \varepsilon_2$$
with probability at least $1-\frac{\delta}{K}$ by Theorem~\ref{Thm-H1} and Theorem~\ref{Thm-H2}. Therefore, with probability at least $1-\delta$ 
\begin{align}
&\|\mu_K - \mu_*\|_1 \nonumber \\
&\leq \sum_{k=0}^{K-1} K_H^{K-(k+1)} \, \|\hat{H}([N,L],M)(\mu_k) - H(\mu_k)\|_1 + \|H^K(\mu_0) - \mu_*\|_1 \nonumber \\
&\leq \sum_{k=0}^{K-1} K_H^{K-(k+1)} \left(K_1 \sqrt{\theta (\varepsilon_1 + \Delta)} + \varepsilon_2\right) + \frac{(K_H)^K}{1-K_H} \nonumber \\
&\leq \frac{K_1 \sqrt{\theta\, \Delta }}{(1-K_H)} + \varepsilon. \nonumber 
\end{align}
This completes the proof.\qed
\endproof

Now, we give the main result of this paper as a corollary of Theorem~\ref{main-result}. It states that, by using learning algorithm $\hat{H}$, one can obtain approximate regularized-mean-field equilibrium policy with high confidence. Since approximate regularized mean-field equilibrium policy constitutes an approximate Nash equilibrium for the finite-agent game model with sufficiently many agents, this learning algorithm also provides approximate Nash equilibrium.

\begin{corollary}\label{main-cor}
	Fix any $(\varepsilon,\delta) \in (0,1)^2$. Suppose that $K,L,N,M$ satisfy the conditions in Theorem~\ref{main-result}. Let $\mu_K$ be the output of the learning algorithm established by random operator $\hat{H}$ with inputs $$\left(K, \{[N,L]\}_{k=0}^K, \{M\}_{k=0}^{K-1}, \mu_0 \right).$$
	Define 
	$$\pi_K(x) \coloneqq \argmax_{u \in \sU} Q_K(x,u),$$ 
	where $Q_{K} = \hat{H}_1([N,L])(\mu_K)$. Then, with probability at least $1-\delta(1+\frac{1}{2K})$, the policy $\pi_K$ is a $\kappa(\varepsilon,\Delta)$-regularized mean-field equilibrium policy, where
	\begin{equation*}
		\kappa(\varepsilon,\Delta) = \sqrt{\theta \, \left(\frac{(1-K_H)^2 \, \varepsilon^2}{16 \, \theta \, (K_1)^2} + \Delta + K_{H_1} \left(\frac{K_1 \sqrt{\theta\, \Delta }}{(1-K_H)} + \varepsilon\right)\right)};
	\end{equation*}
	that is 
	$$
	\sup_{x \in \sX} \|\pi_K(x)-\pi_*(x)\|_1 \leq \kappa(\varepsilon,\Delta).
	$$
	Therefore, with probability at least $1-\delta(1+\frac{1}{2K})$, by Theorem~\ref{old-main-cor}, an $N$-tuple of policies ${\bf \pi}^{(N)} = \{\pi_K,\pi_K,\ldots,\pi_K\}$ is an $\tau \, \kappa(\varepsilon,\Delta) + \sigma$-Nash equilibrium for the regularized game with $N \geq N(\sigma)$ agents if $\mu_0 \in \Lambda^{\reg}(\pi_K)$.
\end{corollary}

\proof
By Theorem~\ref{main-result}, with probability at least $1-\delta(1+\frac{1}{2K})$, we have 
\begin{align}
\|Q_K - H_1(\mu_*)\|_{\infty} &\leq \|Q_K - H_1(\mu_K)\|_{\infty} + \|H_1(\mu_K) - H_1(\mu_*)\|_{\infty} \nonumber \\
&\leq \varepsilon_1 + \Delta + K_{H_1} \|\mu_K - \mu_*\|_1 \nonumber \\
&\leq \varepsilon_1 + \Delta + K_{H_1} \left(\frac{K_1 \sqrt{\theta\, \Delta }}{(1-K_H)} + \varepsilon\right) \nonumber \\
&= \frac{(1-K_H)^2 \, \varepsilon^2}{16 \, \theta \, (K_1)^2} + \Delta + K_{H_1} \left(\frac{K_1 \sqrt{\theta\, \Delta }}{(1-K_H)} + \varepsilon\right). \nonumber
\end{align}

Let $\pi_K(x) \coloneqq \argmax_{u \in \sU} Q_K(x,u)$. Using the same analysis that leads to (\ref{perturbation2}), we can obtain the following bound:
\begin{align}
	\sup_{x \in \sX} \|\pi_K(x)-\pi_*(x)\|^2_1 &\leq \theta \, \|Q_K-H_1(\mu_*)\|_{\infty}. \nonumber 
\end{align}
	Hence, with probability at least $1-\delta(1+\frac{1}{2K})$, the policy $\pi_K$ is a $\kappa(\varepsilon,\Delta)$-regularized mean-field equilibrium, where
\begin{equation*}
	\kappa(\varepsilon,\Delta) = \sqrt{\theta \, \left(\frac{(1-K_H)^2 \, \varepsilon^2}{16 \, \theta \, (K_1)^2} + \Delta + K_{H_1} \left(\frac{K_1 \sqrt{\theta\, \Delta }}{(1-K_H)} + \varepsilon\right)\right)}. 
\end{equation*}
This completes the proof.\qed
\endproof

\begin{remark}\label{remaa}
	In Corollary~\ref{main-cor}, there is a constant error $\Delta$, which is a function of representation error $E({\mathcal F})$. If we choose the class of $Q$-functions ${\mathcal F}$ as ${\mathcal C}$, then there will be no representation error, i.e, $E({\mathcal F})=0$, and so, $\Delta = 0$. Hence, in this case, we have the following error bound:
	$$
	\kappa(\varepsilon,0) \coloneqq \sqrt{\theta \, \left( \frac{(1-K_H)^2 \, \varepsilon^2}{16 \, \theta \, (K_1)^2} +  K_{H_1} \varepsilon \right)},$$
	which goes to zero as $\varepsilon \rightarrow 0$.  
\end{remark}

\section{Numerical Example}\label{num_ex}
	
In this section, we show the effectiveness of the learning algorithm with a numerical example. In this example, we consider a mean-field game with a binary state space $\sX = \{0,1\}$ and a binary action space $\sA = \{0,1\}$. The transition probability $p: \sX \times \sA \rightarrow \Pnew(\sX)$ is independent of the mean-field term and is given by 
\begin{align}
	&p(1|0,0) = \eta, \,\, p(1|1,0) = 1-\alpha, \nonumber \\
	&p(1|0,1)=\kappa, \,\, p(1|1,1) = 1-\xi. \nonumber
\end{align}
The one-stage reward function $r:\sX \times \sA \times \Pnew(\sX) \rightarrow [0,\infty)$ depends on the mean-field term and is defined as 
$$
	r(x,a,\mu) = \tau (1-\langle \mu \rangle)(1-x)+\lambda\langle \mu \rangle (1-a), 
$$
where $\langle \mu \rangle$ is the mean of the distribution $\mu$ on $\sX$. This model satisfies Assumption~\ref{as1} with 
\begin{align}
L_1 &= \max\left\{\tau,\frac{\tau+\lambda}{2}\right\} \nonumber \\ 
K_1 &=  \max\left\{2|1-\alpha-\eta|,|\eta-\kappa|,\frac{2}{3} |1-\xi-\eta|,\frac{2}{3} |1-\alpha-\kappa|,2|1-\xi-\kappa|,|\xi-\alpha|\right\}. \nonumber
\end{align}

In the equivalent game model, the action space becomes $\sU = \Pnew(\sA)$. With this new action space, the new transition probability $P: \sX \times \sU \rightarrow \Pnew(\sX)$ is given by 
$$
	P(\cdot|x,u) = p(\cdot|x,0) \, u(0) + p(\cdot|x,1) \, u(1),
$$
and the new one-stage reward function $R:\sX \times \sU \times \Pnew(\sX) \rightarrow [0,\infty)$ is given by
$$
	R(x,u,\mu) = \tau (1-\langle \mu \rangle)(1-x)+\lambda\langle \mu \rangle u(0).
$$
The regularization function $\Omega: \Pnew(\sX) \rightarrow \R$ is taken as the weighted negative binary entropy:
$$
	\Omega(u) = \gamma \, \left(\log(u(0)) \, u(0) + \log(u(1)) \, u(1) \right). 
$$
Therefore, the regularized one-stage reward function is $$R^{\reg}(x,u,\mu) = R(x,u,\mu)-\Omega(u).$$ 
Note that $\Omega$ is a $\gamma$-strongly convex function with respect to the $l_1$-norm.
	
For numerical results, we use the following values of the parameters:
\begin{align}
	\eta&=0.6, \, \alpha = 0.3, \, \kappa=0.7, \, \xi=0.2 \nonumber \\
	\tau&=0.2, \, \lambda=0.2, \, \gamma=0.15, \, \beta=0.2. \nonumber 
\end{align}
With these parameters, Lipschitz constants in Assumption~\ref{as1} become $L_1=0.2$ and $K_1=0.2$. Using these constants, $\rho=\gamma=0.15$, and $\beta=0.2$, one can also verify that Assumption~\ref{as2} holds.
We run the learning algorithm $20$ times using the following parameters: $N=1000, L=10, M=1000, K=20$ and take the average of the outputs. Here, output of the learning algorithm contains the mean-field term, mean-field policy, and corresponding value function. In fitted $Q$-learning algorithm, we pick the function class $\F$ as two-layer neural networks with $20$ hidden units. We use neural network fitting tool of MATLAB. In particular, we use `$\mathrm{fitnet}$', `$\mathrm{train}$', and `$\mathrm{net}$' functions of MATLAB, where `Levenberg-Marquardt' is picked as the training algorithm and the transfer function is chosen as `hyperbolic tangent sigmoid transfer function'. The parameters of the neural network fitting tool of MATLAB are set to default values.  We also run the value iteration algorithm using MFE operator $H$ to find the correct mean-field term, mean-field policy, and corresponding value function. Then, we compare the learned outputs with correct outputs. Figures~\ref{meanfield_comp}, \ref{policy_comp}, and \ref{reward_comp} show this comparison. It can be seen that learned outputs converge to the true outputs. 
	
\begin{figure}[H]
	\includegraphics[scale=0.5]{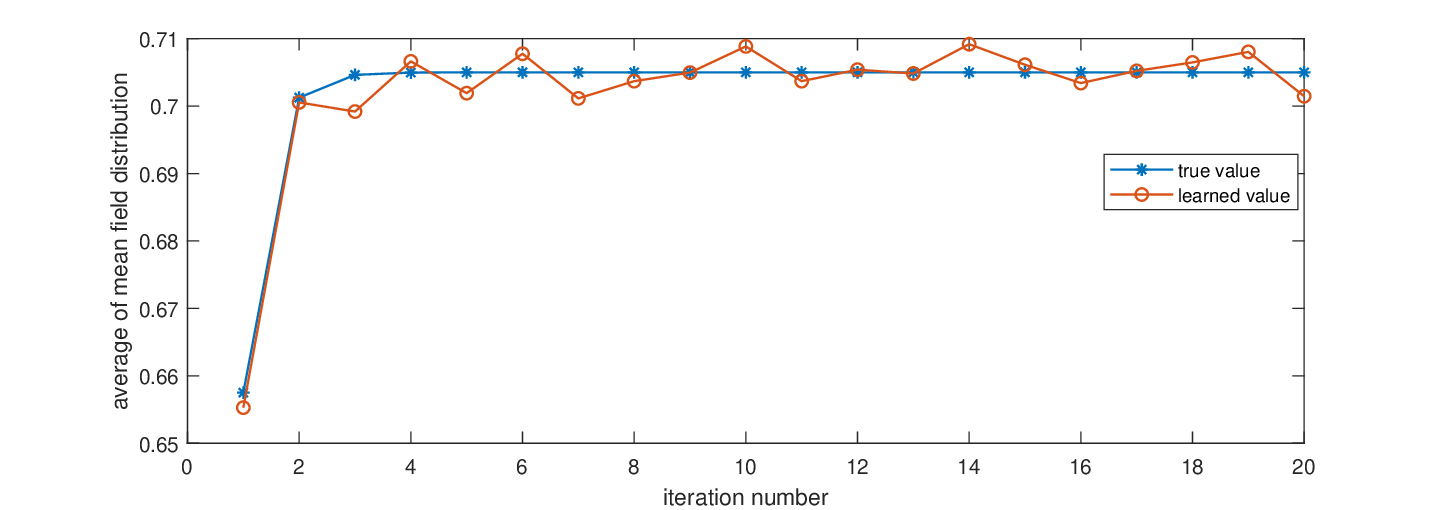}
	\caption{Comparison of mean-field terms} \label{meanfield_comp}
\end{figure}
	
\begin{figure}[H]
	\includegraphics[scale=0.45]{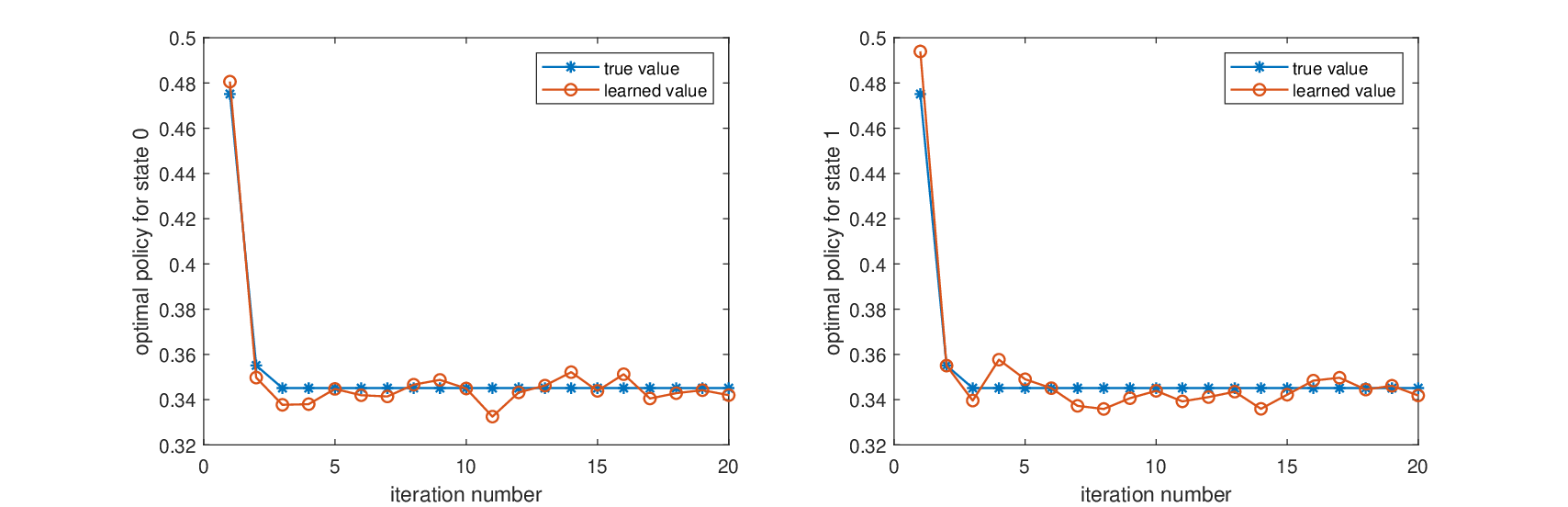}
	\caption{Comparison of policies} \label{policy_comp}
\end{figure}
	
\begin{figure}[H]
	\centering
	\includegraphics[scale=0.45]{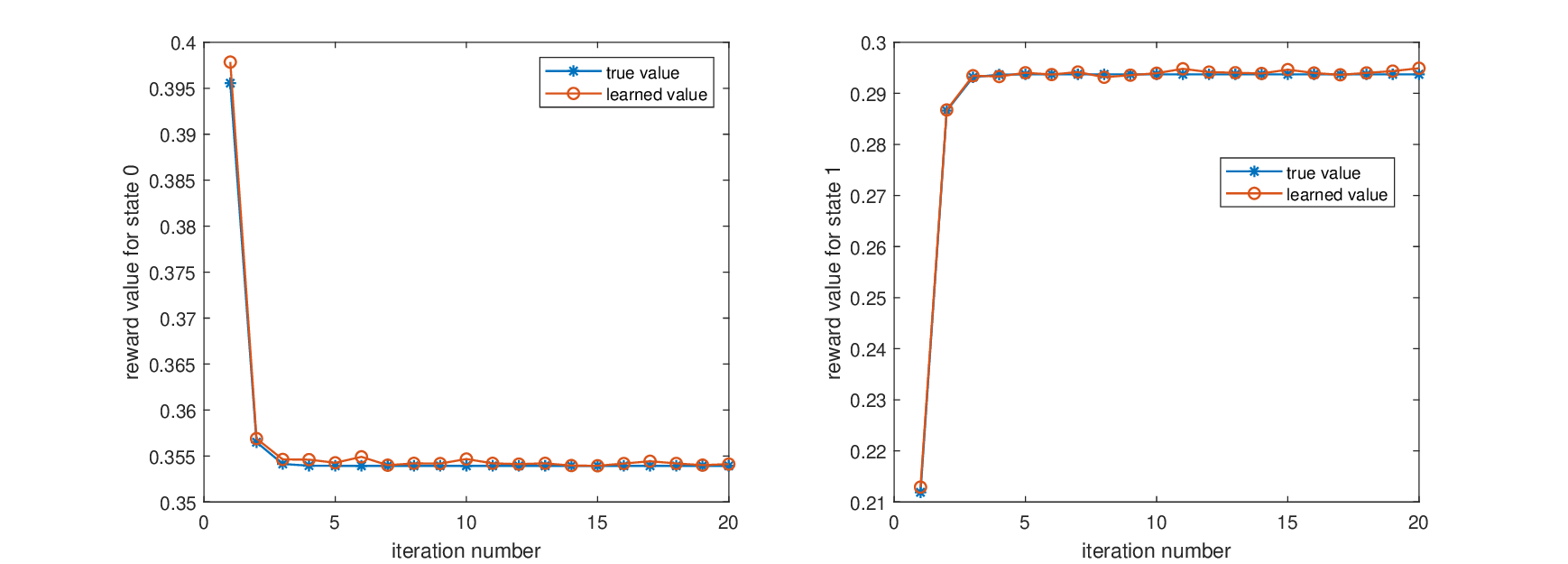}
	\caption{Comparison of rewards} \label{reward_comp}
\end{figure}   

\section{Conclusion}\label{conc}

In this paper, we have established a learning algorithm for discrete time regularized mean-field games subject to discounted reward criterion via fitted Q-learning. It is supposed that adding regularization term to the one-stage reward function makes the learning algorithm more robust and improves exploration. In addition to these advantages, with regularization term, the error analysis of the learning algorithm  has been established under milder assumptions compared to the classical version of the game model. 

One interesting future direction is to study learning regularized mean-field games with abstract observation and action spaces. In this case, to obtain similar results, one needs to extend duality of strong convexity and smoothness to the functions defined on infinite dimensional spaces such as the set of probability measures on abstract spaces.

\section{Appendix}

\subsection{Duality of Strong Convexity and Smoothness}\label{dualitysec}
	
	Suppose that $\sE = \R^d$ for some $d \geq 1$ with an inner product $\langle \cdot,\cdot \rangle$. We denote $\R^{*} = \R \, \bigcup \, \{\infty\}$. Let $f:\sE \rightarrow \R^{*}$ be a differentiable convex function with the domain $S \coloneqq \{x \in \sE: f(x) \in \R\}$, which is necessarily convex subset of $\sE$. The Fenchel conjugate of $f$ is a convex function $f^*:\sE \rightarrow \R^*$ that is defined as 
	$$
	f^*(y) \coloneqq \sup_{x \in S} \, \langle x,y \rangle - f(x).
	$$

	Now, we will state duality result between strong convexity and smoothness. To this end, we suppose that $f$ is $\rho$-strongly convex with respect to a norm $\|\cdot\|$ on $\sE$  (not necessarily Euclidean norm); that is, for all $x,y \in S$, we have 
	$$
	f(y) \geq f(x) + \langle \nabla f(x),y-x \rangle + \frac{1}{2} \rho \|y-x\|^2. 
	$$
	To state the result, we need to define the dual norm of $\|\cdot\|$. The dual norm $\|\cdot\|_*$ of $\|\cdot\|$ on $\sE$ is defined as 
	$$
	\|z\|_* \coloneqq \sup\{\langle z,x \rangle: \|x\| \leq 1\}.
	$$ 
	For example, $\|\cdot\|_{\infty}$ is the dual norm of $\|\cdot\|_{1}$.  
	
	\smallskip
	
	\begin{proposition}[{\cite[Lemma 15]{Sha07}}]\label{duality}
		Let $f:\sE \rightarrow \R^*$ be a differentiable $\rho$-strongly convex function with respect to the norm $\|\cdot\|$ and let $S$ denote its domain. Then
		\begin{itemize}
			\item[1.] $f^*$ is differentiable on $\sE$.
			\item[2.] $\nabla f^*(y) = \argmax_{x \in S} \langle x,y \rangle - f(x)$. 
			\item[3.] $f^*$ is $\frac{1}{\rho}$-smooth with respect to the norm $\|\cdot\|_*$ ; that is, 
			$$
			\|\nabla f^*(y_1) - \nabla f^*(y_2)\| \leq \frac{1}{\rho} \|y_1-y_2\|_* \,\, \text{for all} \,\, y_1,y_2 \in \sE.
			$$
		\end{itemize}
	\end{proposition}
	
	In the paper, we make use of the properties $2$ and $3$ of Proposition~\ref{duality} to establish the Lipschitz continuity of the optimal policies, which enables us to prove the main results of our paper.

\subsection{Proof of Proposition~\ref{new-con}}\label{app0}

Fix any $x,\hat{x}, u,\hat{u}, \mu, \hat{\mu}$. 
		Let us recall the following fact about $l_1$ norm on the set probability distributions on finite sets \cite[p. 141]{Geo11}. Suppose that there exists a real valued function $F$ on a finite set $\sE$. Let $\lambda(F) \coloneqq \sup_{e \in \sE} F(e) - \inf_{e \in \sE} F(e)$. Then, for any pair of probability distributions $\mu,\nu$ on $\sE$, we have 
		\begin{align}
		\left|\sum_{e} F(e) \, \mu(e) - \sum_{e} F(e) \, \nu(e) \right| \leq \frac{\lambda(F)}{2} \, \|\mu-\nu\|_{1}. \label{tv-bound}
		\end{align}
		Using this fact, we now have
		\begin{align*}
		|R(x,u,\mu) - R(\hat{x},\hat{u},\hat{\mu})| &= \left| \sum_{a \in \sA} r(x,a,\mu) \, u(a) - \sum_{a \in \sA} r(\hat{x},a,\hat{\mu}) \, \hat{u}(a) \right| \\
		&\leq  \left| \sum_{a \in \sA} r(x,a,\mu) \, u(a) - \sum_{a \in \sA} r(x,a,\mu) \, \hat{u}(a) \right| 
		\\ &\phantom{xx} + \left| \sum_{a \in \sA} r(x,a,\mu) \, \hat{u}(a) - \sum_{a \in \sA} r(\hat{x},a,\hat{\mu}) \, \hat{u}(a) \right|   \\
		&\leq  L_1 \, \left( 1_{\{x \neq \hat{x}\}} + \|u-\hat{u}\|_1 +\|\mu-\hat{\mu}\|_1 \right), 
		\end{align*}
		where the last inequality follows from the following fact in view of (\ref{tv-bound}):
		\begin{align}
		\sup_{a} r(x,a,\mu) - \inf_{a} r(x,a,\mu) &\coloneqq r(x,a_{\max},\mu) - r(x,a_{\min},\mu) \nonumber \\
		&\leq 2 L_1 \, 1_{\{a_{\max} \neq a_{\min}\}} = 2L_1. \nonumber
		\end{align}

	Similarly, we have
	\begin{align}
	\|P(\cdot|x,u,\mu) &- P(\cdot|\hat{x},\hat{u},\hat{\mu})\|_1 \nonumber \\
	&= \sum_{y \in \sX} \left| P(y|x,u,\mu) - P(y|\hat{x},\hat{u},\hat{\mu}) \right| \nonumber \\
	&=  \sum_{y \in \sX} \left| \sum_{a \in \sA} p(y|x,a,\mu) \, u(a) - \sum_{a \in \sA} p(y|\hat{x},a,\hat{\mu}) \, \hat{u}(a) \right| \nonumber \\
	&\leq \sum_{y \in \sX} \left| \sum_{a \in \sA} p(y|x,a,\mu) \, u(a) - \sum_{a \in \sA} p(y|x,a,\mu) \, \hat{u}(a) \right| \nonumber \\
	&\phantom{xx}+ \sum_{y \in \sX} \left| \sum_{a \in \sA} p(y|x,a,\mu) \, \hat{u}(a) - \sum_{a \in \sA} p(y|\hat{x},a,\hat{\mu}) \, \hat{u}(a) \right| \nonumber \\
	&\overset{(I)}{\leq} K_1 \|u-\hat{u}\|_1  +  \sum_{y \in \sX} \left| \sum_{a \in \sA} p(y|x,a,\mu) \, \hat{u}(a) - \sum_{a \in \sA} p(y|\hat{x},a,\hat{\mu}) \, \hat{u}(a) \right|  \nonumber \\
	&\leq K_1 \, \left( 1_{\{x \neq \hat{x}\}} + \|u-\hat{u}\|_1 +\|\mu-\hat{\mu}\|_1 \right). \nonumber
	\end{align}
	To show that (I) follows from Assumption~\ref{as1}-(b), let us define the transition probability $M:\sA \rightarrow \Pnew(\sX)$ as
	$$
	M(\cdot|a) \coloneqq p(\cdot|x,a,\mu).
	$$
	Let $\xi \in \Pnew(\sA\times\sA)$ be the optimal coupling of $u$ and $\hat{u}$ that achieves total variation distance $\|u-\hat{u}\|_{TV}$. Similarly, for any $a,\hat{a} \in \sA$, let $K(\cdot|a,\hat{a}) \in \Pnew(\sX\times\sX)$ be the optimal coupling of $M(\cdot|a)$ and $M(\cdot|\hat{a})$ that achieves total variation distance $\|M(\cdot|a)-M(\cdot|\hat{a})\|_{TV}$. Note that 
	\begin{equation*}
	\sum_{y \in \sX} \left| \sum_{a \in \sA} p(y|x,a,\mu) \, u(a) - \sum_{a \in \sA} p(y|x,a,\mu) \, \hat{u}(a) \right| = 2 \|u M-\hat{u} M\|_{TV},
	\end{equation*}
	where 
	$$u M(\cdot) \coloneqq \sum_{a \in \sA} M(\cdot|a) \, u(a)$$ 
	and 
	$$\hat{u} M(\cdot) \coloneqq \sum_{a \in \sA} M(\cdot|a) \, \hat{u}(a).$$ 
	Let us define $\nu(\cdot) \coloneqq \sum_{(a,\hat{a}) \in \sA \times \sA} K(\cdot|a,\hat{a}) \, \xi(a,\hat{a})$, and so, $\nu$ is a coupling of $u M$ and $\hat{u} M$. Therefore, we have
	\begin{align}
	2 \, \|u M-\hat{u} M\|_{TV} &\leq 2 \sum_{(x,y) \in \sX\times\sX} 1_{\{x \neq y\}} \, \nu(x,y) \nonumber \\
	&= 2 \sum_{(a,\hat{a}) \in \sA \times \sA} \sum_{(x,y) \in \sX \times \sX} 1_{\{x \neq y\}} \, K(x,y|a,\hat{a}) \, \xi(a,\hat{a}) \nonumber \\
	&= \sum_{(a,\hat{a}) \in \sA \times \sA} \|M(\cdot|a)-M(\cdot|\hat{a})\|_{1} \, \xi(a,\hat{a}) \nonumber \\
	&\leq 2 \, K_1 \, \sum_{(a,\hat{a}) \in \sA \times \sA} 1_{\{a \neq \hat{a}\}} \, \xi(a,\hat{a}) \nonumber \\
	&= K_1 \, \|u-\hat{u}\|_{1}. \nonumber
	\end{align}
	Hence, (I) follows. This completes the proof.

\subsection{Proof of Lemma~\ref{lip-value}}\label{app01}

Fix any $\mu$. If a function $f: \sX \rightarrow \R$ is $K$-Lipschitz continuous for some $K$, then $g = \frac{f}{K}$ is $1$-Lipschitz continuous. Hence, for all $u \in \sU$ and $z,y \in \sX$ we have
		\begin{align}
		\biggl | \sum_{x} f(x) P(x|z,u,\mu) &- \sum_{x} f(x) P(x|y,u,\mu) \biggr | \nonumber \\
		&= K \biggl | \sum_{x} g(x) P(x|z,u,\mu) - \sum_{x} g(x) P(x|y,u,\mu) \biggr | \nonumber \\
		&\leq \frac{K}{2} \, \|P(\,\cdot\,|z,u,\mu) - P(\,\cdot\,|y,u,\mu)\|_1 \nonumber \, \text{(by (\ref{tv-bound}))}\\
		&\leq \frac{KK_1}{2} \, 1_{\{z \neq y\}}, \, \text{(by Proposition~\ref{new-con})} \nonumber
		\end{align}
		since $\sup_x g(x) - \inf_x g(x) \leq 1$. Hence, the contraction operator $T_{\mu}$ maps $K$-Lipschitz functions to $L_1+\beta K K_1/2$-Lipschitz functions, since, for all $z,y \in \sX$
		\begin{align}
		| T_{\mu}f(z) - T_{\mu}f(y) | 
		&\leq \sup_{u} \biggl \{ |R(z,u,\mu) - R(y,u,\mu)| \nonumber \\
		&\phantom{xx}+ \beta \biggl | \sum_{x} f(x) P(x|z,u,\mu) - \sum_{x} f(x) P(x|y,u,\mu) \biggr | \biggr \}\nonumber \\
		&\leq L_1 1_{\{z\neq y\}} + \beta \frac{K K_1}{2} 1_{\{z \neq y\}} = \biggl(L_1 + \beta \frac{K K_1}{2}\biggr) 1_{\{z \neq y\}}. \nonumber
		\end{align}
		Now we apply $T_{\mu}$ recursively to obtain the sequence $\{T_{\mu}^n f\}$ by letting $T_{\mu}^n f = T_{\mu} (T_{\mu}^{n-1} f )$, which converges to the value function $Q^{\reg,*}_{\mu,\max}$ by the Banach fixed point theorem. Clearly, by mathematical induction, we have for all $n\geq1$, $T_{\mu}^n f$ is $K_n$-Lipschitz continuous, where $K_n = L_1 \sum_{i=0}^{n-1} (\beta K_1/2)^i + K (\beta K_1/2)^n$. If we choose $K < L_1$, then $K_n \leq K_{n+1}$ for all $n$ and therefore, $K_n \uparrow \frac{ L_1}{1-\beta K_1/2}$. Hence, $T_{\mu}^n f$ is $ \frac{L_1}{1-\beta K_1/2}$-Lipschitz continuous for all $n$, and therefore, $Q^{\reg,*}_{\mu,\max}$ is also  $ \frac{L_1}{1-\beta K_1/2}$-Lipschitz continuous.

\subsection{Proof of Lemma~\ref{n-lemma1}}\label{app02}

Under Assumption~\ref{as1}, it is straightforward to prove that $H_1$ maps $\Pnew(\sX)$ into ${\mathcal C}$. Indeed, the only non-trivial fact is the $\left(K_{\Lip}+L_{\reg}\right)$-Lipschitz continuity of $H_1(\mu) \eqqcolon Q_{\mu}^{\reg,*}$. This can be proved as follows: For any $(x,u)$ and $(\hat{x},\hat{u})$, we have
	\begin{align}
	|Q_{\mu}^{\reg,*}(x,u)-Q_{\mu}^{\reg,*}(\hat{x},\hat{u})| 
	&= |R(x,u,\mu) - \Omega(u) + \beta \sum_{y} Q_{\mu,\max}^{\reg,*}(y) P(y|x,u,\mu) \nonumber \\
	& - R(\hat{x},\hat{u},\mu) + \Omega(\hat{u}) - \beta \sum_{y} Q_{\mu,\max}^{\reg,*}(y) P(y|\hat{x},\hat{u},\mu)| \nonumber \\
	&\leq L_1 (1_{\{x\neq\hat{x}\}}+\|u-\hat{u}\|_1) + L_{\reg} \|u-\hat{u}\|_1 \nonumber \\
	&\phantom{xx}+ \beta \frac{K_1 K_{\Lip}}{2} \, (1_{\{x\neq\hat{x}\}}+\|u-\hat{u}\|_1),\nonumber   
	\end{align}
	where the last inequality follows from (\ref{tv-bound}) and Lemma~\ref{lip-value}. Hence, $Q_{\mu}^{\reg,*}$ is $\left(K_{\Lip}+L_{\reg}\right)$-Lipschitz continuous. 
	
	Now, for any $\mu,\hmu \in \Pnew(\sX)$, we have 
	
	\begin{align}
	\|H_1(\mu) -& H_1(\hmu)\|_{\infty} = \|Q_{\mu}^{\reg,*}-Q_{\hmu}^{\reg,*}\|_{\infty} \nonumber \\
	&= \sup_{x,u} \bigg| R(x,u,\mu) + \beta \sum_{y} Q_{\mu,\max}^{\reg,*}(y) P(y|x,u,\mu) \nonumber \\
	&\phantom{xxxxxxxxxxxxxx} - R(x,u,\hmu) - \beta \sum_{y} Q_{\hmu,\max}^{\reg,*}(y)  P(y|x,u,\hmu) \bigg| \nonumber \\
	&\leq L_1 \, \|\mu-\hmu\|_1 \nonumber \\
	&\phantom{xx} + \beta \left| \sum_{y} Q_{\mu,\max}^{\reg,*}(y) P(y|x,u,\mu) - \sum_{y} Q_{\mu,\max}^{\reg,*}(y) P(y|x,u,\hmu) \right| \nonumber \\
	&\phantom{xx} + \beta \left| \sum_{y} Q_{\mu,\max}^{\reg,*}(y) P(y|x,u,\hmu) - \sum_{y} Q_{\hmu,\max}^{\reg,*}(y) P(y|x,u,\hmu) \right| \nonumber \\
	&\leq L_1 \, \|\mu-\hmu\|_1 + \frac{\beta  K_1 K_{\Lip}}{2} \, \|\mu-\hmu\|_1 + \beta \, \|Q_{\mu}^{\reg,*}-Q_{\hmu}^{\reg,*}\|_{\infty}, \nonumber
	\end{align}
	
	where the last inequality follows from (\ref{tv-bound}) and Lemma~\ref{lip-value}. This completes the proof.
	
\subsection{Proof of Lemma~\ref{Lips-pol}}\label{app03}

For any $\mu \in \Pnew(\sX)$, we have
	\begin{align}
		Q_{\mu}^{\reg,*}(x,u) &= L_{\mu} Q_{\mu}^{\reg,*}(x,u) \nonumber \\
		&= R(x,u,\mu) + \beta \sum_{y \in \sX} Q_{\mu,\max}^{\reg,*}(y) \, P(y|x,u,\mu) - \Omega(u) \nonumber \\
		&= \langle q_{x}^{\mu},u \rangle - \Omega(u), \nonumber   
	\end{align}
	where 
	$
	q_{x}^{\mu}(\cdot) \coloneqq r(x,\cdot,\mu) + \beta \sum_{y \in \sX} Q_{\mu,\max}^{\reg,*}(y) \, p(y|x,\cdot,\mu).
	$
	By $\rho$-strong convexity of $\Omega$, $Q_{\mu}^{\reg,*}(x,\cdot)$ has a unique maximizer $f_{\mu}(x) \in \sU$ for any $x \in \sX$, which is the optimal policy for $\mu$. By Property $2$ of Proposition~\ref{duality}, we have 
	 $$f_{\mu}(x) = \nabla \Omega^*(q_x^{\mu}),$$ where $\Omega^*$ is the Fenchel conjugate of $\Omega$, and $\Omega^*(q_x^{\mu}) = Q_{\mu,\max}^{\reg,*}(x)$.
		
	Moreover, for any $\mu,\hmu \in \Pnew(\sX)$ and $x,\hx \in \sX$, by property $3$ of Proposition~\ref{duality} and by noting the fact that $\|\cdot\|_{\infty}$ is the dual norm of $\|\cdot\|_1$ on $\sU$, we obtain the following bound:
	\begin{align}
		\|f_{\mu}(x)-f_{\hmu}(\hx)\|_1 \leq \frac{1}{\rho} \, \|q_{x}^{\mu}-q_{\hx}^{\hmu}\|_{\infty}. \nonumber
	\end{align} 
	Note that we have
	\begin{align}
		\|q_{x}^{\mu}-q_{\hx}^{\hmu}\|_{\infty} 
		&=\sup_{a \in \sA} \bigg|r(x,a,\mu) + \beta \sum_{y \in \sX} Q_{\mu,\max}^{\reg,*}(y) \, p(y|x,a,\mu) \nonumber \\
		&\phantom{xxxxxxxxxxxxxxxxx}- r(\hx,a,\hmu) - \beta \sum_{y \in \sX} Q_{\hmu,\max}^{\reg,*}(y) \, p(y|\hx,a,\hmu) \bigg| \nonumber \\
		&\leq L_1 (1_{\{x \neq \hx\}} + \|\mu-\hmu\|_1) \nonumber \\
		&\phantom{xx} + \beta \sup_{a \in \sA} \bigg| \sum_{y} Q_{\mu,\max}^{\reg,*}(y) p(y|x,a,\mu) - \sum_{y} Q_{\hmu,\max}^{\reg,*}(y) p(y|x,a,\mu) \bigg| \nonumber \\
		&\phantom{xx} + \beta \sup_{a \in \sA} \bigg| \sum_{y} Q_{\hmu,\max}^{\reg,*}(y) p(y|x,a,\mu) - \sum_{y} Q_{\hmu,\max}^{\reg,*}(y) p(y|\hx,a,\hmu) \bigg| \nonumber \\
		&\leq L_1 (1_{\{x \neq \hx\}} + \|\mu-\hmu\|_1)  + \beta \| Q_{\mu}^{\reg,*}- Q_{\hmu}^{\reg,*}\|_{\infty} \nonumber \\
		&\phantom{xx} + \beta \frac{K_1 K_{\Lip}}{2} (1_{\{x \neq \hx\}} + \|\mu-\hmu\|_1) \nonumber \\
		&\leq K_{\Lip} \, (1_{\{x \neq \hx\}} + \|\mu-\hmu\|_1) + \beta K_{H_1} \, \|\mu-\hmu\|_1 \nonumber \\
		&\leq K_{H_1} (1_{\{x \neq \hx\}} + \|\mu-\hmu\|_1). \nonumber 
	\end{align}
	Therefore we obtain 
	\begin{align}
		\|f_{\mu}(x)-f_{\hmu}(\hx)\|_1 \leq \frac{1}{\rho} \, K_{H_1} (1_{\{x \neq \hx\}} + \|\mu-\hmu\|_1). \nonumber
	\end{align}

\subsection{Proof of Theorem~\ref{old-main-cor}}\label{app1}
	Let $\mu_{\varepsilon} \in \Lambda^{\reg}(\pi_{\varepsilon})$. Then, we have
	\begin{align}
	&\|\mu_{\varepsilon}-\mu_*\|_1 \nonumber \\ 
	&= \sum_{y} \, \bigg| \sum_{x} \, P(y|x,\pi_{\varepsilon},\mu_{\varepsilon}) \, \mu_{\varepsilon}(x) - \sum_{x} \, P(y|x,\pi_*(x),\mu_*) \, \mu_*(x) \biggr| \nonumber \\
	&\leq \sum_{y} \, \bigg| \sum_{x} \, P(y|x,\pi_{\varepsilon},\mu_{\varepsilon}) \, \mu_{\varepsilon}(x) - \sum_{x} \, P(y|x,\pi_*(x),\mu_*) \, \mu_{\varepsilon}(x) \biggr| \nonumber \\
	&\phantom{xx}+ \sum_{y} \, \bigg| \sum_{x} \,P(y|x,\pi_*(x),\mu_*) \, \mu_{\varepsilon}(x) - \sum_{x} \, P(y|x,\pi_*(x),\mu_*) \, \mu_*(x) \biggr| \nonumber \\
	&\overset{(I)}{\leq} \sum_{x} \left\|P(\cdot|x,\pi_{\varepsilon}(x),\mu_{\varepsilon})-P(\cdot|x,\pi_*(x),\mu_*) \right\|_1 \mu_{\varepsilon}(x) \nonumber \\
	&\phantom{xx}+ \frac{K_1}{2} \left( 1 + \frac{K_{H_1}}{\rho} \right) \, \|\mu_{\varepsilon}-\mu_*\|_1 \nonumber \\
	&\leq K_1 \left( \sup_x \|\pi_{\varepsilon}(x)-\pi_*(x)\|_1 + \|\mu_{\varepsilon}-\mu_*\|_1 \right) \nonumber \\
	&\phantom{xx}+ \frac{K_1}{2} \left( 1 + \frac{K_{H_1}}{\rho} \right) \, \|\mu_{\varepsilon}-\mu_*\|_1 \nonumber \\
	&\leq K_1 \, \varepsilon +  \left(\frac{3 \, K_1}{2}  + \frac{K_1 \, K_{H_1}}{2\rho} \right) \, \|\mu_{\varepsilon}-\mu_*\|_1. \nonumber 
	\end{align}
	Note that Lemma~\ref{Lips-pol} and Proposition~\ref{new-con} lead to
	\begin{equation*}
	\|P(\cdot|x,\pi_*(x),\mu_*)-P(\cdot|y,\pi_*(y),\mu_*)\|_1 \leq K_1 \left( 1 + \frac{K_{H_1}}{\rho} \right) \, 1_{\{x \neq y\}}.
	\end{equation*}
	Hence, (I) follows from \cite[Lemma A2]{KoRa08}. Therefore, we have: 
	$$
	\|\mu_{\varepsilon}-\mu_*\|_1 \leq \frac{K_1 \, \varepsilon}{1-C_1},
	$$ 
	where $C_1 \coloneqq \left(\frac{3 \, K_1}{2}  + \frac{K_1 \, K_{H_1}}{2\rho} \right)$. Note that by Assumption~\ref{as2}, $C_1 < 1$. Now, fix any policy $\pi \in \Pi$. Then, we have
	\begin{align}
	&\|J_{\mu_*}^{\reg}(\pi,\cdot)-J_{\mu_{\varepsilon}}^{\reg}(\pi,\cdot)\|_{\infty} \nonumber \\
	&=\sup_{x} \bigg| R^{\reg}(x,\pi(x),\mu_*) + \beta \, \sum_{y} J_{\mu_*}^{\reg}(\pi,y) \, p(y|x,\pi(x),\mu_*) \nonumber \\
	&\phantom{xx}-R^{\reg}(x,\pi(x),\mu_{\varepsilon}) - \beta \, \sum_{y} J_{\mu_{\varepsilon}}^{\reg}(\pi,y) \, p(y|x,\pi(x),\mu_{\varepsilon})\bigg| \nonumber \\
	&\leq L_1 \, \|\mu_*-\mu_{\varepsilon}\|_1 \nonumber \\
	&\phantom{xx} + \beta \sup_{x} \bigg|\sum_{y} J_{\mu_*}^{\reg}(\pi,y) \, p(y|x,\pi(x),\mu_*) - \sum_{y} J_{\mu_*}^{\reg}(\pi,y) \, p(y|x,\pi(x),\mu_{\varepsilon})\bigg| \nonumber \\
	&\phantom{xx} + \beta \sup_{x} \bigg|\sum_{y} J_{\mu_*}^{\reg}(\pi,y) \, p(y|x,\pi(x),\mu_{\varepsilon}) - \sum_{y} J_{\mu_{\varepsilon}}^{\reg}(\pi,y) \, p(y|x,\pi(x),\mu_{\varepsilon})\bigg| \nonumber \\
	&\overset{(II)}{\leq} \left(L_1+\frac{\beta K_1 K_{\Lip}}{2}\right) \|\mu_*-\mu_{\varepsilon}\|_1 + \beta \|J_{\mu_*}^{\reg}(\pi,\cdot)-J_{\mu_{\varepsilon}}^{\reg}(\pi,\cdot)\|_{\infty} \nonumber \\
	&\leq  \left(L_1+\frac{\beta K_1 K_{\Lip}}{2}\right) \frac{K_1  \varepsilon}{1-C_1}+ \beta \|J_{\mu_*}^{\reg}(\pi,\cdot)-J_{\mu_{\varepsilon}}^{\reg}(\pi,\cdot)\|_{\infty}. \nonumber 
	\end{align}
	Here (II) follows from (\ref{tv-bound}) and the fact that $J_{\mu_*}^{\reg}(\pi,\cdot)$ is $K_{\Lip}$-Lipschitz continuous, which can be proved as in Lemma~\ref{lip-value}.  
	Therefore, we obtain 
	\begin{align}\label{nneqq1}
	\|J_{\mu_*}^{\reg}(\pi,\cdot)-J_{\mu_{\varepsilon}}^{\reg}(\pi,\cdot)\|_{\infty} \leq \frac{C_2 \, \varepsilon}{1-\beta}, 
	\end{align}
	where $C_2 \coloneqq \left(L_1+\frac{\beta K_1 K_{\Lip}}{2}\right) \frac{K_1}{1-C_1}$.

	Note that we also have 
	\begin{align}
	&\|J_{\mu_*}^{\reg}(\pi_*,\cdot)-J_{\mu_*}^{\reg}(\pi_{\varepsilon},\cdot)\|_{\infty} \nonumber \\
	&=\sup_{x} \bigg| R^{\reg}(x,\pi_*(x),\mu_*) + \beta \, \sum_{y} J_{\mu_*}^{\reg}(\pi_*,y) \, p(y|x,\pi_*(x),\mu_*) \nonumber \\
	&\phantom{xx}-R^{\reg}(x,\pi_{\varepsilon}(x),\mu_*) - \beta \, \sum_{y} J_{\mu_*}^{\reg}(\pi_*,y) \, p(y|x,\pi_{\varepsilon}(x),\mu_*)\bigg| \nonumber \\
	&\leq (L_1 + L_{\reg}) \, \sup_{x} \|\pi_*(x)-\pi_{\varepsilon}(x)\|_1 \nonumber \\
	&\phantom{xx} + \beta \sup_{x} \bigg|\sum_{y} J_{\mu_*}^{\reg}(\pi_*,y) \, p(y|x,\pi_*(x),\mu_*) - \sum_{y} J_{\mu_*}^{\reg}(\pi_*,y) \, p(y|x,\pi_{\varepsilon}(x),\mu_*)\bigg| \nonumber \\
	&\phantom{xx} + \beta \sup_{x} \bigg|\sum_{y} J_{\mu_*}^{\reg}(\pi_*,y) \, p(y|x,\pi_{\varepsilon}(x),\mu_*) - \sum_{y} J_{\mu_*}^{\reg}(\pi_{\varepsilon},y) \, p(y|x,\pi_{\varepsilon}(x),\mu_*)\bigg| \nonumber \\
	&\overset{(III)}{\leq}  \left(L_1+L_{\reg}+\frac{\beta K_1 K_{\Lip}}{2}\right) \sup_x \|\pi_*(x)-\pi_{\varepsilon}(x)\|_1 \nonumber \\
	&\phantom{xxxx}+ \beta \|J_{\mu_*}^{\reg}(\pi_*,\cdot)-J_{\mu_*}^{\reg}(\pi_{\varepsilon},\cdot)\|_{\infty} \nonumber \\
	&\leq  \left(L_1+L_{\reg}+\frac{\beta K_1 K_{\Lip}}{2}\right) \, \varepsilon+ \beta \|J_{\mu_*}^{\reg}(\pi_*,\cdot)-J_{\mu_*}^{\reg}(\pi_{\varepsilon},\cdot)\|_{\infty}. \nonumber 
	\end{align}
	Here (III) follows from (\ref{tv-bound}) and the fact that $J_{\mu_*}^{\reg}(\pi_*,\cdot)$ is $K_{\Lip}$-Lipschitz continuous, which can be proved as in Lemma~\ref{lip-value}.
	Therefore, we obtain 
	\begin{align}\label{nneqq2}
	\|J_{\mu_*}^{\reg}(\pi_*,\cdot)-J_{\mu_*}^{\reg}(\pi_{\varepsilon},\cdot)\|_{\infty} \leq \frac{C_3\varepsilon}{1-\beta}, 
	\end{align}
	where $C_3 \coloneqq \left(L_1+L_{\reg}+\frac{\beta K_1 K_{\Lip}}{2}\right)$. 
	
	Note that we must prove that
	\begin{align}
	J_i^{(N)}({\boldsymbol \pi}^{(N)}) &\geq \sup_{\pi^i \in \Pi_i} J_i^{(N)}({\boldsymbol \pi}^{(N)}_{-i},\pi^i) - \tau \, \varepsilon - \delta \label{old-eq13}
	\end{align}
	for each $i=1,\ldots,N$, when $N$ is sufficiently large. As the transition probabilities and the one-stage reward functions are the same for all agents, it is sufficient to prove (\ref{old-eq13}) for Agent~$1$ only. Given $\delta > 0$, for each $N\geq1$, let $\tpi^{(N)} \in \Pi_1$ be such that
	\begin{align}
	J_1^{(N)} (\tpi^{(N)},\pi_{\varepsilon},\ldots,\pi_{\varepsilon}) > \sup_{\pi' \in \Pi_1} J_1^{(N)} (\pi',\pi_{\varepsilon},\ldots,\pi_{\varepsilon}) - \frac{\delta}{3}. \nonumber
	\end{align}
	Then, by \cite[Theorem 4.10]{SaBaRaSIAM}, we have
	\begin{align}
	&\lim_{N\rightarrow\infty} J_1^{(N)} (\tpi^{(N)},\pi_{\varepsilon},\ldots,\pi_{\varepsilon}) = \lim_{N\rightarrow\infty} J_{\mu_{\varepsilon}}^{\reg}(\tpi^{(N)}) \nonumber \\
	&\leq \lim_{N\rightarrow\infty} J_{\mu_*}^{\reg}(\tpi^{(N)}) + \frac{C_2 \, \varepsilon}{1-\beta} \,\, \text{(by (\ref{nneqq1}))}\nonumber \\
	&\leq \sup_{\pi'} J_{\mu_*}^{\reg}(\pi') + \frac{C_2 \, \varepsilon}{1-\beta} \nonumber \\
	&= J_{\mu_*}^{\reg}(\pi_*) + \frac{C_2 \, \varepsilon}{1-\beta} \nonumber \\
	&\leq J_{\mu_*}^{\reg}(\pi_{\varepsilon}) + \frac{C_2 \, \varepsilon}{1-\beta} + \frac{C_3 \, \varepsilon}{1-\beta}\,\, \text{(by (\ref{nneqq2}))}\nonumber \\
	&\leq J_{\mu_{\varepsilon}}^{\reg}(\pi_{\varepsilon}) + \frac{2 \, C_2 \, \varepsilon}{1-\beta} + \frac{C_3 \, \varepsilon}{1-\beta}\,\, \text{(by (\ref{nneqq1}))}\nonumber \\
	&= \lim_{N\rightarrow\infty} J_1^{(N)} (\pi_{\varepsilon},\pi_{\varepsilon},\ldots,\pi_{\varepsilon}) + \tau \,\varepsilon\nonumber. \nonumber
	\end{align}
	Therefore, there exists $N(\delta)$ such that
	\begin{align}
	&\sup_{\pi' \in \Pi_1} J_1^{(N)} (\pi',\pi_{\varepsilon},\ldots,\pi_{\varepsilon}) - \delta - \tau \,\varepsilon \nonumber \\
	&\leq J_1^{(N)} (\tpi^{(N)},\pi_{\varepsilon},\ldots,\pi_{\varepsilon}) - \frac{2\delta}{3} - \tau \, \varepsilon \nonumber \\
	&\leq J_{\mu_*}^{\reg}(\pi_{\varepsilon}) - \frac{\delta}{3}   \nonumber \\
	&\leq J_1^{(N)} (\pi_{\varepsilon},\pi_{\varepsilon},\ldots,\pi_{\varepsilon}). \nonumber
	\end{align}
	for all $N\geq N(\delta)$.


\begin{acknowledgements}
This work was partly supported by the BAGEP Award of the Science Academy.
\end{acknowledgements}

%
%

\end{document}